\documentclass[11pt]{amsart}
\usepackage{amsmath,amssymb,mathrsfs,color}
\usepackage[title]{appendix}
\usepackage{hyperref}
\hypersetup{hypertex=true,
colorlinks=true,
linkcolor=blue,
anchorcolor=blue,
citecolor=blue}
\usepackage{amsmath,amssymb,mathrsfs,color}
\usepackage{tikz}

\newtheorem{lemma}{Lemma}[section]
\newtheorem{theorem}[lemma]{Theorem}
\newtheorem{remark}[lemma]{Remark}
\newtheorem{prop}[lemma]{Proposition}

\newtheorem{definition}[lemma]{Definition}

\def\N{{\mathbb N}}

\def\R{{\mathbb R}}
\def\C{{\mathbb C}}

\def\P{{\mathbb P}}
\def\E{{\mathbb E}}
\def\mbE{{\mathbb E}}
\def\mbP{{\mathbb P}}
\def\T{{\mathbb T}}
\def\mcL{{\mathcal L}}

\def\veps{\varepsilon}
\def\la{\langle}
\def\ra{\rangle}
\def\d{{\rm d}}

\allowdisplaybreaks
\numberwithin{equation}{section}

\title[Averaging principle for stochastic CGL equations]
{Averaging principle for stochastic complex
Ginzburg-Landau equations}

\author{mengyu Cheng}
\address{M. Cheng: School of Mathematical Sciences,
Dalian University of Technology, Dalian 116024, P. R. China;
Fakultat f\"ur Mathematik,
Universit\"at Bielefeld, D-33501 Bielefeld, Germany}
\email{mengyu.cheng@hotmail.com;
mengyucheng@mail.dlut.edu.cn}

\author{zhenxin Liu}
\address{Z. Liu:
School of Mathematical Sciences,
Dalian University of Technology, Dalian 116024, P. R. China}
\email{zxliu@dlut.edu.cn}

\author{Michael R\"ockner}
\address{M. R\"ockner: Fakultat f\"ur Mathematik,
Universit\"at Bielefeld, D-33501 Bielefeld, Germany}
\email{roeckner@math.uni-bielefeld.de}

\date{October 10, 2022}
\begin{document}

%%%%%%%%%%%%%%%%%%%%%%%%%%%%%%
%%%%%%%%%%%%%%%%%%%%%%%%%%%%%%
%%%%%%%%%%
\begin{abstract}
Averaging principle is an effective method for investigating
dynamical systems with highly oscillating components.
In this paper, we study three types of averaging principle
for stochastic complex Ginzburg-Landau equations.
Firstly, we prove that the solution of the original equation
converges to that of the averaged equation on finite intervals
as the time scale $\veps$ goes to zero
when the initial data are the same.
Secondly, we show that there exists a unique recurrent solution
(in particular, periodic, almost periodic, almost automorphic,
etc.) to the original equation in a neighborhood of the stationary
solution of the averaged equation when the time scale is small.
Finally, we establish the global averaging
principle in weak sense, i.e. we show that the
attractor of original system tends to that of the averaged
equation in probability measure space as $\veps$
goes to zero.
\end{abstract}

\maketitle

\section{Introduction}

In this paper, we consider the following stochastic complex  Ginzburg-Landau
(in short, CGL) equation with highly oscillating components
on the $d$-torus $\T^{d}, d=1,2,3$
\begin{equation}\label{in:CGL}
\d u^{\veps}(t)=\left[(1+i\alpha )\Delta u^{\veps}(t)
-(\gamma(t/\veps)+i\beta)|u^{\veps}(t)|^{2}u^{\veps}(t)
+f\left(t/\veps,u^{\veps}(t)\right)\right]\d t
+g\left(t/\veps,u^{\veps}(t)\right)\d W(t),
\end{equation}
where $\alpha,\beta\in\R$,
$\veps$ is a small parameter and $W(t),t\in\R$ is a two-sided cylindrical
Wiener process with the identity covariance operator on a separable
Hilbert space $\left(U,\la~,~\ra_U\right)$.
Here $\gamma$, $f$ and $g$ satisfy some suitable conditions.

We note that $\gamma(\frac{t}{\veps})$, $f(\cdot,\frac{t}{\veps})$ and $g(\cdot,\frac{t}{\veps})$
are fast time oscillating forces that depend on the solution.
Usually, studying the original system \eqref{in:CGL} is relatively difficult,
since there are two widely separated timescales.
It is well-known that a highly oscillating system may be ``averaged''
under some suitable conditions, and the resulting averaged system
is easier to analyze and governs the evolution of the original system
over long time scales. This is the basic idea of averaging principle.
According to the connotation of approximations,
there are three types of interpretation for averaging
principle, i.e. the so-called {\em first Bogolyubov theorem},
{\em second Bogolyubov theorem} and
{\em global averaging principle}.

More specifically,
the first Bogolyubov theorem requires that the solution of
the original equation \eqref{in:CGL} converges,
as $\veps\rightarrow0$, to that of
the averaged equation
\begin{equation}\label{in:aCGL}
\d \bar{u}(t)=\left[(1+i\alpha )\Delta \bar{u}(t)-(\bar{\gamma}+i\beta)|\bar{u}(t)|^{2}\bar{u}(t)
+\bar{f}(\bar{u}(t))\right]\d t+\bar{g}\left(\bar{u}(t)\right)\d W(t)
\end{equation}
on finite time intervals when $u^{\veps}(0)=\bar{u}(0)$.
And the second Bogolyubov theorem states that
the stationary solution of \eqref{in:aCGL}
approximates the periodic solution of \eqref{in:CGL}; that is to say,
the approximation is valid on the entire real axis.
So sometimes it is called {\em theorem for periodic solution by averaging}.
In addition, the global averaging principle describes that
the attractor of \eqref{in:CGL} tends to that of the averaged
equation \eqref{in:aCGL} as $\veps$ goes to zero.

The theory of averaging has been applied in many fields,
such as celestial mechanics, oscillation theory and radiophysics.
And the idea of averaging dates from the perturbation theory
which was developed by Clairaut, Laplace and Lagrange in the 18th
century. Then a fairly rigorous averaging method for nonlinear
oscillations was established by Krylov, Bogolyubov and
Mitropolsky \cite{KB1943, BM1961}, which is called the
Krylov-Bogolyubov method nowadays. After that,
averaging principle for finite and infinite dimensional
deterministic systems was studied by many authors which
we will not mention here.

Meanwhile, Stratonovich firstly proposed the stochastic averaging
method based on physical considerations. Then a mathematical proof
was given by Khasminskii \cite{Khas1968}. Following Khasminskii's work,
extensive works concerning averaging principle for
finite and infinite dimensional stochastic differential equations
were conducted; see e.g. \cite{BK2004, Cerr2009, CF2009,
CL2017, DW2014, FW2006, FW2012, Gao2022, HL2020, Kifer2004,  LRSX2020,
MSV1991, RX2021, Skor1989, SXX2021, Vere1990, Vere1999, Vrko1995, WR2012} and
references therein. It should be pointed out that the above existing results
focus on the first Bogolyubov theorem.

To the best of our knowledge, there are few works on
averaging principle for stochastic CGL equations.
As discussed in \cite{Gao2021, HKM2015, Kuk2013},
averaging method was developed to describe the behavior
of solutions for small oscillations
in damped/driven Hamiltonian systems.

The complex Ginzburg-Landau equation arises in physics. Therefore, it
has a very rich physical background and connotations. It can be used
to describe problems of Bandard convection, Taylor-Couette flow,
plane Poiseuille flow, and chemical turbulence. It has also been
applied in superfluidity and superconductivity theory (see e.g.
\cite{AK2002} for more information).

As we know, some perturbations may be neglected in the derivation of
the ideal model. When considering the perturbation of each microscopic
unit of the model, which will lead to a very large complex system,
we usually represent micro effects by random perturbations in
the dynamics of macro observables. Thus, it is more realistic to
consider stochastic CGL equations.

%$\T^{d}, d=1,2,3$
%\begin{equation}\label{maineq1}
%\d u(t)=\left[(1+i\alpha )\Delta u(t)-(1+i\beta)|u(t)|^{2}u(t)
%+f(t,u(t))\right]\d t + g(t,u(t))\d W(t).
%\end{equation}

From the perspective of theoretical and practical value, we establish
three types of averaging principle for the stochastic CGL
equation with highly oscillating components
in this paper following \cite{CLAM2021}. Firstly, under some suitable
conditions, employing the technique of time discretization
which is also used in
\cite{Cerr2009, CF2009, CL2017, LRSX2020}, we show that
\begin{equation}\label{FBTeq}
\lim\limits_{\veps\rightarrow0} \mathbb E\left(
\sup_{s\leq t\leq s+T}\|u_{\veps}(t,s,\zeta^{\veps}_{s})
-\bar{u}(t,s,\zeta_{s})\|_{L^2(\T^d)}^{2}\right)=0
\end{equation}
for all $s\in\R$ and $T>0$ provided
$\lim\limits_{\veps\rightarrow0} \mathbb E
\|\zeta^{\veps}_{s}-\zeta_{s}\|_{L^2(\T^d)}^{2}=0$,
where $u_{\veps}$ is the solution of \eqref{in:CGL}
with the initial condition
$u_{\veps}(s,s,\zeta^{\veps}_{s})=\zeta^{\veps}_{s}$ and $\bar{u}$ is
the solution of the averaged equation \eqref{in:aCGL}
with the initial condition $\bar{u}(s,s,\zeta_{s})=\zeta_{s}$.
Here $\bar{\gamma}$, $\bar{f}$ and $\bar{g}$ satisfy
\begin{equation}\label{KBMgam}
\left|\frac{1}{T}\int_{t}^{t+T}\gamma(s)\d s-\bar{\gamma}\right|\leq\delta_\gamma(T),
\end{equation}
\begin{equation}\label{KBMf}
\left\|\frac{1}{T}\int_{t}^{t+T}f(s,x)\d s-\bar{f}(x)\right\|_{L^2(\T^d)}
\leq\delta_{f}(T)(1+\|x\|_{L^2(\T^d)})
\end{equation}
and
\begin{equation*}\label{KBMg}
\frac{1}{T}\int_{t}^{t+T}\left\|g(s,x)-\bar{g}(x)\right\|_{L_{2}(U,L^{2}(\T^{d}))}^{2}\d s
\leq\delta_{g}(T)(1+\|x\|^2_{L^2(\T^d)})
\end{equation*}
for all $t\in\R$ and $x\in L^2(\T^d)$, where
$\delta_\gamma(T)\rightarrow0$, $\delta_{f}(T)\rightarrow0$ and
$\delta_{g}(T)\rightarrow0$ as $T\rightarrow\infty$.
We write $L_{2}(U,H)$ to mean the space of Hilbert-Schmidt operators from
$U$ to $H$. Notice that this is the first Bogolyubov theorem; see Theorem \ref{avethf}.

It can be verified that \eqref{KBMf} implies
\begin{equation*}
  \bar{f}(x)=\lim_{T\rightarrow\infty}\frac{1}{T}\int_t^{t+T}
  f(s,x)\d s
\end{equation*}
uniformly for all $t\in\R$ and $x$ in any bounded subset
of $L^2(\T^d)$.
Such $f$ is called a {\em KBM-vector field} (KBM
stands for Krylov, Bogolyubov and Mitropolsky);
see e.g. \cite{SVM2007}.
We note that $\gamma$ in the cubic term of \eqref{in:CGL} depends on time $t$, which is more general and cannot be covered by the framework in \cite{CLAM2021}
because there it is assumed that the coefficients involved in the averaging principle
are globally Lipschitz continuous.
We employ the interpolation inequality and stopping time techniques to
establish the first Bogolyubov theorem.
Therefore, the method is very different from \cite{CLAM2021}.

%We note that condition \eqref{KBMgam} states the relation between the cubic
%coefficients of the original equation and averaged equation, which cannot be covered by the framework in \cite{CLAM2021}
%because it is assumed that the coefficients involved in the averaging principle are globally Lipschitz continuous there.
%To deal with this problem for the first Bogolyubov theorem, we need the estimates of solutions to
%\eqref{in:CGL} and \eqref{in:aCGL} for the integral of time increment on $L^4(\T^d)$ in Lemma \ref{orstes}
%and the technique of truncation.

After that, we establish the second Bogolyubov theorem for stochastic CGL equations.
To be specific, we firstly show that there exists a unique
$\mcL^{2}(\Omega,\mathbb P;L^{2}(\T^{d}))$-bounded solution
$u^{\veps}(t),t\in\R$ of \eqref{in:CGL} which
inherits the recurrent properties (in particular, periodic,
quasi-periodic, almost periodic, almost automorphic, etc.) of the coefficients
in distribution sense for each $0<\veps\leq1$.
This result is interesting in its own right,
because recurrence is an important concept in dynamical systems,
which roughly means that a motion returns infinitely often to any
small neighborhood of the initial position.
We also prove that the
$\mcL^{2}(\Omega,\mathbb P;L^{2}(\T^{d}))$-bounded solution
$u^{\veps}(t),t\in\R$ to \eqref{in:CGL} is globally
asymptotically stable in square-mean sense. Then we obtain
\begin{equation}\label{result1}
\lim_{\veps\rightarrow0}\sup_{s\leq t\leq s+T}W_{2}(\mathscr{L}(u^{\veps}(t)),\mathscr{L}(\bar{u}(0)))=0
\end{equation}
for all $s\in\R$ and $T>0$ (see Theorem \ref{averth}), where $\bar{u}$ is the unique stationary
solution of \eqref{in:aCGL} and $\mathscr L(\bullet)$ is the
distribution of $\bullet$. Here $W_{2}$ is the Wasserstein
distance.

Little work has been done on the second Bogolyubov
theorem for stochastic differential equations.
For this purpose, recall that the second Bogolyubov theorem for
stochastic ordinary differential equations was studied
in \cite{CL2020}. Compared to \cite{KMF2015},
we consider stochastic CGL equations that admit polynomial growth terms.
Despite that a general second Bogolyubov theorem
was established in \cite{CLAM2021}, it
seems difficult to apply directly to stochastic CGL equations.
Indeed, in order to obtain the recurrent solution and the second Bogolyubov theorem,
we need the tightness of distributions of the bounded solution.
To this end, a condition (denoted by (H6) in \cite{CLAM2021}) was introduced to obtain that the $\mcL^2(\Omega,\mathbb P;H)$-bounded solution is also
$\mcL^2(\Omega,\mathbb P;S)$-bounded in \cite{CLAM2021},
where $S\subset H$ is compact and $V\subset H\subset V^*$ is a Gelfand triple.
However, when we consider stochastic CGL equations,
(H6) is too strong and hard to verify.
In this paper, to bypass (H6), we use the Galerkin method to
obtain that the $\mcL^2(\Omega,\mathbb P;L^2(\T^d))$-bounded solution of the stochastic CGL equation
is indeed $\mcL^2(\Omega,\mathbb P;H^1(\T^d))$-bounded,
which implies the tightness of distributions of this bounded solution
by the compact embedding of
$H_0^1(\T^d)\subset L^2(\T^d)$.

Our another aim of this paper is to establish the global averaging
principle in weak sense for stochastic CGL equations.
The global averaging principle was investigated for deterministic systems;
see e.g. \cite{HV1990, Ily1996, Ily1998, Zel2006} among others.
But to our knowledge, there is only one work on the global averaging
principle for stochastic equations, i.e. \cite{CLAM2021}.

Let $u^{\veps}(t,s,x)$ be the solution of \eqref{in:CGL}
with the initial data $u^{\veps}(s,s,x)=x$
for all $0<\veps\leq1$. Then it generates a Markov transition probability
\[
P_{\veps}(s,x,t,\d y):=\mathbb P\circ\left(u^{\veps}(t,s,x)\right)^{-1}(\d y).
\]
And it acts on the space of probability measures $Pr(L^{2}(\T^{d}))$ by
\[
P_{\veps}^{*}(t,F,\mu)(B):=
\int_{L^{2}(\T^{d})}P_{\veps}
(0,x,t,B)\mu(\d x)
\]
for all $\mu\in Pr(L^{2}(\T^{d}))$ and
$B\in\mathcal B(L^{2}(\T^{d}))$, where $F:=(\gamma,f,g)$ is
as in \eqref{in:CGL}, and $Pr(L^2(\T^d))$ is the space of probability measures on $L^2(\T^d)$.
By the method in \cite{CLAM2021}, we can prove that $P_{\veps}^{*}$ is a cocycle
over $(\mathcal H(F_{\veps}),\R,\sigma)$ with fiber
$Pr_{2}(L^{2}(\T^{d}))$ for any $0<\veps\leq1$, where
$(\mathcal H(F),\R,\sigma)$ is a shift dynamical system
(see Appendix \ref{shiftDS} for details),
$F_{\veps}:=(\gamma(\cdot/\veps),f(\cdot/\veps,\cdot),g(\cdot/\veps,\cdot))$ and
$$
Pr_{2}(L^{2}(\T^{d})):=\left\{\mu\in
Pr(L^{2}(\T^{d})):\int_{L^{2}(\T^{d})}\|z\|^{2}
\mu(\d z)<\infty\right\}.
$$
Finally, we show that $P_{\veps}^{*}$ has a uniform attractor
$\mathcal{A}^{\veps}$ in $Pr_{2}(L^{2}(\T^{d}))$
for any $0<\veps\leq1$, and
\[
 \lim_{\veps\rightarrow0}{\rm dist}_{Pr_{2}(L^{2}
 (\T^{d}))}\left(
 \mathcal{A}^{\veps},\bar{\mathcal{A}}\right)=0
\]
provided $\mathcal H(F)$ is compact (see Theorem \ref{gath}), where
${\rm dist}_{Pr_{2}(L^{2}(\T^{d}))}$ is the Hausdorff
semi-metric and $\bar{\mathcal{A}}:=\{\mathscr L(\bar{u}(0))\}$ is the
attractor of $\bar{P}^{*}$ to the averaged equation \eqref{in:aCGL}.
Note that $\mathcal H(F)$ is compact provided $F$
is Birkhoff recurrent.

We point out that there is an additional assumption on high regularity of initial data
to study the first Bogolyubov theorem due to the cubic term involved in averaging (see Theorem \ref{avethf}).
We also note that the first Bogolyubov theorem plays an important role in establishing
the second Bogolyubov theorem and the global averaging principle. As mentioned in \cite{CLAM2021},
the required high regularity of the bounded solution could not be obtained
under the general monotone framework considered there.
In the case of stochastic CGL equations, we prove that the $\mcL^2(\Omega,\mathbb P;L^2(\T^d))$-bounded solution
is actually $\mcL^{2p}(\Omega,\mathbb P;H_0^1(\T^d))\cap\mcL^2(\Omega,\mathbb P;H_0^2(\T^d))$-bounded for some $p>1$;
this regularity is high enough to establish the second Bogolyubov theorem and global averaging principle.
This is another main novelty of this paper.

Now we introduce the structure of the paper.
In section \ref{Wellp}, we give the well-posedness of
stochastic CGL equation.
In section \ref{firstBT}, we study the first Bogolyubov theorem
for stochastic CGL equations.
In section \ref{secondBT}, we firstly prove that there exists a unique
$\mcL^{2}(\Omega,\P;L^{2}(\T^{d}))$-bounded solution
which possesses the same recurrent properties as the coefficients
in distribution sense and this bounded solution is globally
asymptotically stable in square-mean sense. Then we establish the
second Bogolyubov theorem for stochastic CGL equations.
In section \ref{globalAP}, we prove the global averaging principle for
stochastic CGL equations.
In the Appendix at the end, we recall some definitions of dynamical systems
and some spaces.

\vspace{5mm}
{\bf Notations.}
Throughout this paper, we write $L^{p}(\T^{d};\C),p\geq2$  to mean
the space of all Lebesgue $p$-integrable complex-valued
functions on $\T^{d}, d=1,2,3$.
We view $L^{2}(\T^{d})$ as a real Hilbert space with the inner product
$\langle u,v\rangle:=\langle u,v\rangle_{L^{2}(\T^{d};\C)}
=\mathcal R\int_{\T^{d}}u(\xi)\bar{v}(\xi)\d\xi$, which induce the norm
$\|u\|=\langle u,u\rangle^{\frac{1}{2}}$.
Here $\bar{v}$ is the conjugate of $v$ and
$\mathcal Rv$ is the real part of $v$. Denote by
$H^{m}:=W^{m,2}(\T^{d};\C),m\in\N_+$ the
Sobolev space of complex-valued functions on $\T^{d}$.
Set
\[
H_0^m:=\{u\in H^m: \int_{\T^d}u(\xi)\d\xi=0\}
\]
equipped with the inner product
\[
\la u,v\ra_{m}:=\la(-\Delta)^mu,v\ra
\]
for all $m\geq0$, which induce the norm $\|u\|_m=\la(-\Delta)^m u,u\ra_m^{\frac12}$,
where $H_0^0:=\{u\in L^2(\T^d): \int_{\T^d}u(\xi)\d\xi=0\}$.
In this paper, we write $L^p(\T^d),p\geq2$ to mean $L^p(\T^d;\C)\cap H_0^0$
for simplicity.
Let $\lambda_{*}$ be the first eigenvalue of $-\Delta$
on $L^{2}(\T^{d})$.
Denote by $\{e_{i},i\in\mathbb N\}\subset H_0^{m}$ the eigenfunctions of
$-\Delta$ forming an orthonormal basis of $L^{2}(\T^{d})$ and
$H_{n}:={\rm span}\{e_{1},...,e_{n}\}$.
Let $P_{n}$ be the projection mapping from $L^{2}(\T^{d})$ to $H_{n}$.
Let $H$ be a separable Hilbert space with the norm $\|\cdot\|_{H}$ and
$(\Omega , \mathcal F,\mathbb P)$ a complete probability space.
Then for any $p\geq2$, $\mcL^{p}(\Omega,\mathbb P;H)$
consists of all $H$-valued random variables $X$ such that
$\mbE\|X\|_H^p:=\int_{\Omega}\|X\|_H^p \d\mathbb P<\infty.$
We employ $C_{b}(H)$ to denote the space of all bounded continuous
real-valued functions on $H$.
Denote by $C_{b}(\mathbb R,H)$ the Banach space of all
continuous and bounded mappings $\varphi :\mathbb R\to H$
equipped with the norm
$\|\varphi\|_{\infty}:=\sup\{\|\varphi(t)\|_H:t\in\mathbb R\}$.
For any complete metric space $\mathcal X$,
let $C(\R,\mathcal X)$ be the space of all continuous mappings
$\varphi:\R\rightarrow \mathcal X$ with the compact-open topology.
Fix $\varphi\in C(\R,\mathcal X)$. Define
$\mathcal H(\varphi):=\overline{\{\varphi^{\tau}:\tau\in\R\}}$
and $\varphi^{\tau}(t)=\varphi(t+\tau)$ for all $t\in\R$.
We write $|\cdot|$ to mean the Euclidean norm on $\R$.
Denote by $[C]$ the integer part of $C$ for any constant $C\geq0$.

\section{Well-posedness of stochastic CGL equations}\label{Wellp}

Let $W(t),t\in\R$ be a two-sided cylindrical Wiener process with
the identity covariance operator defined on a separable Hilbert
space $(U,\langle~,~\rangle_{U})$.
We set $\mathcal F_{t}:=\sigma\{W(u)-W(v):  u,v\leq t\}$.
Consider the following stochastic CGL
equation on $\T^{d}, d=1,2,3$
\begin{equation}\label{maineq1}
\d u(t)=\left[(1+i\alpha )\Delta u(t)-(\gamma(t)+i\beta)|u(t)|^{2}u(t)
+f(t,u(t))\right]\d t + g(t,u(t))\d W(t),~t\in\R,
\end{equation}
where $\alpha,\beta\in\R$, $\gamma\in C_b(\R,\R_+)$,
$f:\R\times L^2(\T^d)\rightarrow L^2(\T^d)$
and $g:\R\times L^2(\T^d)\rightarrow L_2(U,L^2(\T^d))$.

Let us introduce the following conditions to investigate
the regularity of solutions to \eqref{maineq1}.
\begin{itemize}
  \item [\textbf{(H$_f^1$)}] There exist constants
  $K,L_{f}>0$
  such that for all $t\in\R$ and $x,y\in L^{2}(\T^{d})$
  $$\|f(t,x)-f(t,y)\|\leq L_{f}\|x-y\|,
  \quad \|f(t,0)\|\leq K.$$
  \item [\textbf{(H$_f^2$)}] There exist constants
  $K,L_{f}>0$
  such that for all $t\in\R$ and $x\in H_0^1$
  $$\|f(t,x)\|_{1}\leq L_{f}\|x\|_{1}+K.$$
  \item [\textbf{(H$_g^1$)}] There exist constants $K,L_{g}>0$ such that
  for all $t\in\R$ and $x,y\in L^{2}(\T^{d})$
  $$\|g(t,x)-g(t,y)\|_{L_{2}(U,L^{2}(\T^{d}))}\leq L_{g}\|x-y\|,
  \quad \|g(t,0)\|_{L_{2}(U,L^{2}(\T^{d}))}\leq K.$$
  \item [\textbf{(H$_g^2$)}] There exist constants $K,L_{g}>0$
  such that for all $t\in\R$ and $x,y\in H_0^{1}$
  $$
  \|g(t,x)-g(t,y)\|_{L_{2}(U,H_0^{1})}
  \leq L_{g}\|x-y\|_{1},\quad
  \|g(t,0)\|_{L_{2}(U,H_0^{1})}\leq K.
  $$
  \item [\textbf{(H$_g^3$)}] There exist constants $K,L_{g}>0$
  such that for all $t\in\R$ and $x,y\in H_0^{2}$
  $$
  \|g(t,x)-g(t,y)\|_{L_{2}(U,H_0^{2})}
  \leq L_{g}\|x-y\|_{2},\quad
  \|g(t,0)\|_{L_{2}(U,H_0^{2})}\leq K.
  $$
\end{itemize}

\begin{remark}\rm
Note that
\[
\|G\|_{L_{2}(U,L^{2}(\T^{d}))}\leq
\|G\|_{L_{2}(U,H_0^{1})}\leq\|G\|_{L_{2}(U,H_0^{2})}
\]
for all $G\in L_{2}(U,H_0^{2})$; see Remark B.0.6 in \cite{LR2015}
for more details.
\end{remark}

\begin{definition}\label{defsol} \rm
Fix $s\in\R$ and $T>0$. A continuous
$L^{2}(\T^{d})$-valued $\mathcal F_{t}$-adapted process
$u(t),s\leq t\leq s+T$ is said to be a {\em solution} of equation
\eqref{maineq1}, if
$$u\in L^{4}\left([s,s+T]\times\Omega,\d t
\otimes\mathbb P;L^{4}(\T^{d})\right)\cap
L^{2}\left([s,s+T]\times\Omega,\d t\otimes\mathbb P;H_0^{1}\right)$$
and it satisfies the following
stochastic integral equation $\mathbb P$-a.s.
\begin{align*}
u(t)
&
=\zeta_{s}+\int_{s}^{t}\left[(1+i\alpha)\Delta u(\tau)
-(\gamma(\tau)+i\beta)
|u(\tau)|^{2}u(\tau)+f(\tau,u(\tau))\right]\d \tau\\
&\quad
+\int_{s}^{t}g(\tau,u(\tau))\d W(\tau),\quad
s\leq t\leq s+T
\end{align*}
for any $\zeta_{s}\in \mcL^{2}\left(\Omega,\mathbb P;
L^{2}(\T^{d})\right)$.
\end{definition}

In what follows, for simplicity we write $C$ instead of
$C_{\lambda_{*},L_{f},L_{g},K,\alpha,\beta,|\gamma|_\infty}$
when $C$ depends on some parameters of
$\lambda_{*},L_{f},L_{g},K, \alpha,\beta$
and $|\gamma|_\infty$. But we write $C_{a}$ explicitly
when $C$ depends on other constant $a$.
We point out that $C$ and $C_{a}$ may change from line to line.
In order to study the well-posedness of \eqref{maineq1},
we will need some estimates.

\begin{lemma}\label{uniest}
Suppose that {\rm(H$_f^1$)} and {\rm(H$_g^1$)} hold. Fix $s\in\R$.
Let $u(t,s,\zeta_{s}),t\geq s$ be the solution of \eqref{maineq1}
with initial data $\zeta_{s}$.
For any $p\geq1$ and $T>0$,
if $\zeta_{s}\in \mcL^{2p}(\Omega,\mathbb P;L^{2}(\T^{d}))$ then
\begin{align}\label{uniest1}
  &
  \mathbb E\left(\sup_{s\leq t\leq s+T}\|u(t,s,\zeta_{s})\|^{2p}\right)
  +\mathbb E\int_{s}^{s+T}\|u(t,s,\zeta_{s})\|^{2p-2}
  \|u(t,s,\zeta_{s})\|_{1}^{2}\d t\\\nonumber
  &\quad
  +\mathbb E\int_{s}^{s+T}\gamma(t)
  \|u(t,s,\zeta_{s})\|^{2p-2}
  \|u(t,s,\zeta_{s})\|_{L^{4}(\T^{d})}^{4}\d t
  \leq C_{T}(1+\mathbb E\|\zeta_{s}\|^{2p}),
\end{align}
where the constant $C_{T}$ depends on
$p,\lambda_{*},L_{f},L_{g},K$ and $T$.
\end{lemma}
\begin{proof}
By It\^o's formula, (H$_f^1$) and (H$_g^1$), we get
\begin{align}\label{equniest1}
&
\|u(t,s,\zeta_{s})\|^{2p}\\\nonumber
&
=\|\zeta_{s}\|^{2p}+\int_{s}^{t}
p\|u(\sigma,s,\zeta_{s})\|^{2p-2}
\Big(2\langle(1+i\alpha )
\Delta u(\sigma,s,\zeta_{s}),u(\sigma,s,\zeta_{s})\rangle\\\nonumber
&\qquad
-2\la(\gamma(\sigma)+i\beta)|u(\sigma,s,\zeta_{s})|^{2}
u(\sigma,s,\zeta_{s}),u(\sigma,s,\zeta_{s})\ra\\\nonumber
&\qquad
+2\langle f(\sigma,u(\sigma,s,\zeta_{s})),u(\sigma,s,\zeta_{s}) \rangle
+\|g(\sigma,u(\sigma,s,\zeta_{s}))\|^{2}_{L_{2}(U,L^{2}(\T^{d}))}
\Big)\d\sigma\\\nonumber
&\quad
+2\int_{s}^{t}p\|u(\sigma,s,\zeta_{s})\|^{2p-2}
\langle u(\sigma,s,\zeta_{s}),g(\sigma,u(\sigma,s,\zeta_{s}))
\d W(\sigma)\rangle\\\nonumber
&\quad
+2p(p-1)\int_s^t\|u(\sigma,s,\zeta_{s})\|^{2p-4}
\|(g(\sigma,u(\sigma,s,\zeta_{s})))^*u(\sigma,s,\zeta_{s})\|_U^2\d\sigma\\\nonumber
&
\leq\|\zeta_{s}\|^{2p}+\int_{s}^{t}
p\|u(\sigma,s,\zeta_{s})\|^{2p-2}\Big(
-2\|u(\sigma,s,\zeta_{s})\|_{1}^{2}
-2\gamma(\sigma)\|u(\sigma,s,\zeta_{s})\|_{L^{4}(\T^{d})}^{4}
+4K^2\\\nonumber
&\qquad
+\theta\|u(\sigma,s,\zeta_{s})\|^{2}\Big)\d\sigma
+2\int_{s}^{t}p\|u(\sigma,s,\zeta_{s})\|^{2p-2}
\langle u(\sigma,s,\zeta_{s}),g(\sigma,u(\sigma,s,\zeta_{s}))
\d W(\sigma)\rangle\\\nonumber
&\quad
+2p(p-1)\int_s^t\|u(\sigma,s,\zeta_{s})\|^{2p-2}
\left(2L_g^2\|u(\sigma,s,\zeta_{s})\|^2+2K^2\right)\d\sigma,\\\nonumber
\end{align}
where $\theta:=2L_{f}^{2}+2L_{g}^{2}+1$.
Dropping negative terms on the right of the above inequality,
it follows from (H$_g^1$), Young's inequality and
Burkholder-Davis-Gundy inequality that
\begin{align*}
&
\mathbb E\left(\sup_{s\leq t\leq s+T}
\|u(t,s,\zeta_{s})\|^{2p}\right)\\
&
\leq \mathbb E\|\zeta_{s}\|^{2p}
+C\mathbb E\left(\sup_{s\leq t\leq s+T}\int_{s}^{t}\left(
\|u(\sigma,s,\zeta_{s})\|^{2p}+1\right)\d\sigma\right)\\
&\quad
+C\mathbb E \left(\int_{s}^{s+T}\|u(\sigma,s,\zeta_{s})\|^{4p-2}
\|g(\sigma,u(\sigma,s,\zeta_{s}))\|^{2}_{L_{2}(U,L^{2}(\T^{d}))}
\d\sigma\right)^{\frac{1}{2}}\\
&
\leq\mathbb E\|\zeta_{s}\|^{2p}+C\mathbb E\int_{s}^{s+T}\left(
\|u(\sigma,s,\zeta_{s})\|^{2p}+1\right)\d\sigma\\
&\quad
+\frac{1}{2}\mathbb E\left(\sup_{s\leq t\leq s+T}
\|u(t,s,\zeta_{s})\|^{2p}\right)+C_T\mathbb E\int_{s}^{s+T}
\|g(\sigma,u(\sigma,s,\zeta_{s}))\|^{2p}_{L_{2}(U,L^{2}(\T^{d}))}
\d\sigma\\
&
\leq\mathbb E\|\zeta_{s}\|^{2p}+C_{T}
+\frac{1}{2}\mathbb E\left(\sup_{s\leq t\leq s+T}
\|u(t,s,\zeta_{s})\|^{2p}\right)
+C_T\mathbb E\int_{s}^{s+T}\sup_{s\leq\tau\leq\sigma}
\|u(\tau,s,\zeta_{s})\|^{2p}\d\sigma.
\end{align*}
In view of Gronwall's lemma, we obtain
\begin{equation}\label{equniest2}
\mathbb E\left(\sup_{s\leq t\leq s+T}\|u(t,s,\zeta_{s})\|^{2p}\right)
\leq\left(2\mathbb E\|\zeta_{s}\|^{2p}+C_{T}\right)
{\rm{e}}^{C_T}.
\end{equation}

Taking expectation and letting $t=s+T$ on \eqref{equniest1}, with the help of
\eqref{equniest2} we have
\begin{align*}
&
\mathbb E\int_{s}^{s+T}\|u(\sigma,s,\zeta_{s})\|^{2p-2}
\|u(\sigma,s,\zeta_{s})\|_{1}^{2}\d\sigma
+\mathbb E\int_{s}^{s+T}\gamma(\sigma)
\|u(\sigma,s,\zeta_{s})\|^{2p-2}
\|u(\sigma,s,\zeta_{s})\|_{L^{4}(\T^{d})}^{4}\d \sigma\\
&
\leq C\mathbb E\|\zeta_{s}\|^{2p}
+C\mathbb E\int_{s}^{s+T}\left(
\|u(\sigma,s,\zeta_{s})\|^{2p}
+1\right)\d \sigma \\
&
\leq C_{T}\left(\mathbb E\|\zeta_{s}\|^{2p}+1\right).
\end{align*}
\end{proof}

In this paper, we mainly consider the stochastic CGL equation with
dissipative cubic term. Now we introduce a condition in the
following lemma, which implies that the cubic term is dissipative.
\begin{lemma}\label{Disscubic}
Let $t\in\R$. If $\gamma(t)\geq\frac{|\beta|}{\sqrt3}$, then
\[
\la(\gamma(t)+i\beta)|u|^2u-(\gamma(t)+i\beta)|v|^2v,u-v\ra
\geq0
\]
for all $u,v\in H_0^1(\T^d)$.
\end{lemma}
\begin{proof}
For any $u=u_1+iu_2,v=v_1+iv_2\in H_0^1({\T^d})$, let ${\bf u}=(u_1,u_2)^T, {\bf v}=(v_1,v_2)^T$,
where $(u_1,u_2)^T$ is the transpose of $(u_1,u_2)$.  Define
$$
H(t,{\bf u}):=|{\bf u}|^{2}\begin{pmatrix} \gamma(t) & -\beta \\
\beta & \gamma(t) \end{pmatrix} {\bf u}.
$$
We note that
\begin{align*}
\la(\gamma(t)+i\beta)|u|^2u-(\gamma(t)+i\beta)|v|^2v,u-v\ra
=\la H(t,{\bf u})-H(t,{\bf v}),{\bf u}-{\bf v}\ra_{L^2(\T^d;\R)\times L^2(\T^d;\R)}.
\end{align*}
Now for any $t\in\R$ and ${\bf a}=(a_1,a_2)^T\in\R^2$ we obtain
\begin{align*}
\partial_{\bf a} H(t,{\bf a})
&
=\begin{pmatrix} 3\gamma(t)a_{1}^{2}-2\beta a_{1}a_{2}+\gamma(t)a_{2}^{2} &
-\beta a_{1}^{2}+2\gamma(t)a_{1}a_{2}-3\beta a_{2}^{2}\\
3\beta a_{1}^{2}+2\gamma(t)a_{1}a_{2}+\beta a_{2}^{2} &
\gamma(t)a_{1}^{2}+2\beta a_{1}a_{2}+3\gamma(t)a_{2}^{2} \end{pmatrix}\\
&
=B(t,{\bf a})+\begin{pmatrix} 0 &-2\beta a_{1}^{2}-2\beta a_{2}^{2}\\
2\beta a_{1}^{2}+2\beta a_{2}^{2} & 0 \end{pmatrix},
\end{align*}
where
$$
B(t,{\bf a})=\begin{pmatrix} 3\gamma(t)a_{1}^{2}-2\beta a_{1}a_{2}+\gamma(t)a_{2}^{2} &
\beta a_{1}^{2}+2\gamma(t)a_{1}a_{2}-\beta a_{2}^{2}\\
\beta a_{1}^{2}+2\gamma(t)a_{1}a_{2}-\beta a_{2}^{2} &
\gamma(t)a_{1}^{2}+2\beta a_{1}a_{2}+3\gamma(t)a_{2}^{2} \end{pmatrix}.
$$
Note that the diagonal elements of $B$ are positive for all
$t\in\R,{\bf a}\in \R^2$ if $\gamma(t)>\frac{|\beta|}{\sqrt{3}}$. And
\begin{equation*}
{\rm det} B(t,{\bf a})=\left(3\gamma(t)^2-\beta^{2}\right)
\left(a_{1}^{2}+a_{2}^{2}\right)^{2}
=\left(3\gamma(t)^2-\beta^{2}\right)|{\bf a}|^{4}
\end{equation*}
is also positive. Thus the matrix $B$ is positive definite
for all $t\in\R,{\bf a}\in\R^2$ if $\gamma(t)>\frac{|\beta|}{\sqrt{3}}$.
Therefore, for all $u,v\in L^2(\T^d)$, we have
\begin{align*}
&
\la(\gamma(t)+i\beta)|u|^2u-(\gamma(t)+i\beta)|v|^2v,u-v\ra\\
&
=\la\partial_{\bf a}H(t,{\bf w})({\bf u}-{\bf v}),{\bf u}-{\bf v}
\ra_{L^2(\T^d;\R)\times L^2(\T^d;\R)}\geq0
\end{align*}
provided $\gamma(t)>\frac{|\beta|}{\sqrt{3}}$ for $t\in\R$,
where ${\bf w}=\theta{\bf u}+(1-\theta){\bf v}$ for some $\theta\in[0,1]$.
\end{proof}

With the help of \eqref{uniest1},
we give a theorem about the existence and uniqueness of solutions
to \eqref{maineq1}. The proof is based on the Galerkin
method (see e.g. \cite{LR2015}), and we put it in
Appendix \ref{eupf}.

\begin{theorem}\label{wellD}
Let $\gamma(\cdot)\geq\frac{|\beta|}{\sqrt3}$.
Assume that {\rm(H$_f^1$)} and {\rm(H$_g^1$)} hold. Let $s\in\R$.
Then for any $\zeta_{s}\in \mcL^{2}(\Omega,\mathbb P;L^{2}(\T^{d}))$
there exists a unique solution $u(t,s,\zeta_{s}),t\geq s$
to \eqref{maineq1} in the sense of Definition \ref{defsol}.
\end{theorem}

Now we discuss the regularity of solutions to \eqref{maineq1}.
\begin{lemma}\label{uniest01}
Let $\gamma(\cdot)\geq\frac{|\beta|}{\sqrt3}$.
Suppose that {\rm(H$_f^1$)} and {\rm(H$_g^1$)} hold. Fix $s\in\R$.
Let $u(t,s,\zeta_{s}),t\geq s$ be the solution of \eqref{maineq1}
with initial data $\zeta_{s}$.
For any $p\geq1$ and $T>0$ we have the following statements.
\begin{enumerate}
  \item If $\zeta_{s}\in \mcL^{2p}(\Omega,\mathbb P;H_0^1)$
  and (H$_g^2$) holds, then
  \begin{equation}\label{uniest2}
  \sup_{s\leq t\leq s+T}\mathbb E
  \|u(t,s,\zeta_{s})\|_{1}^{2p}
  \leq C_{T}
  \left(1+\mathbb E\|\zeta_s\|_{1}^{2p}\right).
  \end{equation}
  \item
  If $\zeta_{s}\in \mcL^{2p}(\Omega,\mathbb P;H_0^2)\cap\mcL^{10p}(\Omega,\mathbb P;H_0^1)$,
  (H$_f^2$) and (H$_g^3$) hold, then
  \begin{equation}\label{uniest3}
  \sup_{s\leq t\leq s+T}\mbE\|u(t,s,\zeta_s)\|_{2}^{2p}
  \leq C_{T}\left(1+\mathbb E\|\zeta_s\|_{2}^{2p}\right).
  \end{equation}
\end{enumerate}
Here the constant $C_{T}$ depends on $p,\lambda_{*},L_{f},L_{g},K$ and $T$.
\end{lemma}

\begin{proof}
(i)
Fixing $n\geq1$, let $u^n(t),t\geq s$ be the solution to
the  finite dimensional equation \eqref{GFeq}.
Then it follows from It\^o's formula and integration by parts that
\begin{align*}
\mathbb E\|u^{n}(t)\|_{1}^{2p}
&
=\mathbb E\|P_{n}\zeta_{s}\|_{1}^{2p}+\mathbb E\int_{s}^{t}
p\|u^{n}(\sigma)\|_{1}^{2p-2}
\Big(
2\langle(1+i\alpha )\Delta u^{n}(\sigma),u^{n}(\sigma)\rangle_{1}\\
&\qquad
-2\la(\gamma(t)+i\beta)P_{n}|u^{n}(\sigma)|^{2}u^{n}(\sigma),
u^{n}(\sigma)\ra_{1}\\
&\qquad
+2\langle P_{n}f(\sigma,u^{n}(\sigma)),
u^{n}(\sigma)\rangle_{1}
+\|P_{n}g(\sigma,u^{n}(\sigma))\|_{L_{2}(U,H_0^{1})}^{2}\Big)\d\sigma\\
&\quad
+2p(p-1)\mathbb E\int_{s}^{t}\|u^{n}(\sigma)\|_{1}^{2p-4}
\|(P_{n}g(\sigma,u^{n}(\sigma)))^*u^{n}(\sigma)\|_U^2\d\sigma\\
&
\leq\mathbb E\|\zeta_{s}\|_{1}^{2p}+\mathbb E\int_{s}^{t}
p\|u^{n}(\sigma)\|_{1}^{2p-2}\Big(
2\langle-(1+i\alpha )\Delta u^{n}(\sigma),\Delta u^{n}(\sigma)\rangle\\
&\qquad
+2\la(\gamma(t)+i\beta)P_{n}|u^{n}(\sigma)|^{2}u^{n}(\sigma),
\Delta u^{n}(\sigma)\ra\\
&\qquad
-2\langle P_{n}f(\sigma,u^{n}(\sigma)),
\Delta u^{n}(\sigma)\rangle
+\|g(\sigma,u^{n}(\sigma))\|_{L_{2}(U,H_0^{1})}^{2}
\Big)\d\sigma\\
&\quad
+2p(p-1)\mathbb E\int_{s}^{t}\|u^{n}(\sigma)\|_{1}^{2p-2}
\|g(\sigma,u^{n}(\sigma))\|_{L_2(U,H_0^1)}^2\d\sigma.
\end{align*}
Notice that
\[
\langle(\gamma(t)+i\beta)|u|^{2}u,\Delta u \rangle\leq0
\]
for all $u\in H^{1}$ and $t\in\R$ if $\gamma(\cdot)\geq0$, which implies that
\begin{align}\label{esteqH1}
&
\mathbb E\|u^{n}(t)\|_{1}^{2p}\\\nonumber
&
\leq\mathbb E\|\zeta_{s}\|_{1}^{2p}+\mathbb E\int_{s}^{t}
p\|u^{n}(\sigma)\|_{1}^{2p-2}\Big(
-2\|\Delta u^{n}(\sigma)\|^{2}
-2\langle f(\sigma,u^{n}(\sigma)),\Delta u^{n}(\sigma)\rangle\\\nonumber
&\qquad
+(2p-1)\|g(\sigma,u^{n}(\sigma))
-g(\sigma,0)\|_{L_{2}(U,H_0^{1})}^{2}
-(2p-1)\|g(\sigma,0)\|_{L_{2}(U,H_0^{1})}^{2}\\\nonumber
&\qquad
+2(2p-1)\langle g(\sigma,u^{n}(\sigma)),
g(\sigma,0)\rangle_{L_{2}(U,H_0^{1})}\Big)
\d\sigma
\end{align}
since $\la P_nv,w\ra=\la v,w\ra$ for all $v\in L^2(\T^d)$ and $w\in H_n$.
In view of (H$_f^1$), (H$_g^2$) and Young's inequality, we obtain
\begin{align*}
\mathbb E\|u^{n}(t)\|_{1}^{2p}
&
\leq\mathbb E\|\zeta_{s}\|_{1}^{2p}+C\mathbb E\int_{s}^{t}
\|u^{n}(\sigma)\|_{1}^{2p-2}\Big(
\|u^{n}(\sigma)\|^{2}
+\|u^{n}(\sigma)\|_{1}^{2}
+1\Big)\d\sigma\\
&
\leq\mathbb E\|\zeta_{s}\|_{1}^{2p}+C\mathbb E\int_{s}^{t}
\left(\|u^{n}(\sigma)\|_{1}^{2p}
+1\right)\d\sigma.
\end{align*}
By Gronwall's lemma, one sees that
\begin{align}\label{uGLH1}
\mathbb E\|u^{n}(t)\|_{1}^{2p}
\leq C_T(1+\E\|\zeta_s\|_{1}^{2p}).
\end{align}
Therefore, it follows from the reflexivity of $\mcL^{2p}(\Omega,\mbP;H_0^1)$
that
\begin{align*}
\sup_{s\leq t\leq s+T}\mathbb E\|u(t)\|_{1}^{2p}
\leq C_T(1+\E\|\zeta_s\|_{1}^{2p}).
\end{align*}

(ii)
By It\^o's formula, we have
\begin{align*}
\mathbb E\|u^{n}(t)\|_{2}^{2p}
&
=\mathbb E\|P_{n}\zeta_{s}\|_{2}^{2p}+\mathbb E\int_{s}^{t}
p\|u^{n}(\sigma)\|_{2}^{2p-2}
\Big(
2\langle(1+i\alpha )\Delta u^{n}(\sigma),u^{n}(\sigma)\rangle_{2}\\
&\qquad
-2\la(\gamma(t)+i\beta)P_{n}|u^{n}(\sigma)|^{2}u^{n}(\sigma),
u^{n}(\sigma)\ra_{2}\\
&\qquad
+2\langle P_{n}f(\sigma,u^{n}(\sigma)),
u^{n}(\sigma)\rangle_{2}
+\|P_{n}g(\sigma,u^{n}(\sigma))\|_{L_{2}(U,H_0^{2})}^{2}\Big)\d\sigma\\
&\quad
+2p(p-1)\mathbb E\int_{s}^{t}\|u^{n}(\sigma)\|_{2}^{2p-4}
\|(P_{n}g(\sigma,u^{n}(\sigma)))^*u^{n}(\sigma)\|_U^2\d\sigma\\
&
\leq\mathbb E\|\zeta_{s}\|_{2}^{2p}+\mathbb E\int_{s}^{t}
p\|u^{n}(\sigma)\|_{2}^{2p-2}\Big(
2\langle(1+i\alpha )\Delta u^{n}(\sigma),(-\Delta)^2 u^{n}(\sigma)\rangle\\
&\qquad
-2\la(\gamma(t)+i\beta)P_{n}|u^{n}(\sigma)|^{2}u^{n}(\sigma),
(-\Delta)^2 u^{n}(\sigma)\ra\\
&\qquad
+2\langle P_{n}f(\sigma,u^{n}(\sigma)),
(-\Delta)^2 u^{n}(\sigma)\rangle
+\|g(\sigma,u^{n}(\sigma))\|_{L_{2}(U,H_0^{2})}^{2}
\Big)\d\sigma\\
&\quad
+2p(p-1)\mathbb E\int_{s}^{t}\|u^{n}(\sigma)\|_{2}^{2p-2}
\|g(\sigma,u^{n}(\sigma))\|_{L_2(U,H_0^2)}^2\d\sigma,
\end{align*}
which implies that
\begin{align*}%\label{H201}
&
\mathbb E\|u^{n}(t)\|_{2}^{2p}\\\nonumber
&
\leq\mathbb E\|\zeta_{s}\|_{2}^{2p}+\mathbb E\int_{s}^{t}
p\|u^{n}(\sigma)\|_{2}^{2p-2}\Big(-2\|u^n(\sigma)\|_{3}^2
+C\|u^n(\sigma)\|_{3}\||u^n(\sigma)|^2|\nabla u^n(\sigma)|\|\\\nonumber
&\qquad
+2\|f(\sigma,u^n(\sigma))\|_{1}\|u^n(\sigma)\|_{3}
+(2p-1)\|g(\sigma,u^{n}(\sigma))\|_{L_2(U,H_0^2)}^2\Big)\d\sigma
\end{align*}
by integration by parts. Note that for all $u\in H_0^3$ we have
\begin{align*}
\|u\|_{3}\||u|^2|\nabla u|\|
&
\leq\|u\|_{3}\|u\|_{L^6(\T^d)}^2\|\nabla u\|_{L^6(\T^d)}\\\nonumber
&
\leq\|u\|_{3}\|u\|_{1}^2\|u\|_{2}\\\nonumber
&
\leq C\|u\|_{3}^{3/2}\|u\|_{1}^{5/2}
\end{align*}
by the embedding of $H_0^1\subset L^6(\T^d),d\leq3$ and the following Gagliardo-Nirenberg inequality
\begin{equation*}%\label{GNeq}
\|u\|_{2}\leq C\|u\|_{3}^{1/2}\|u\|^{1/2}_{L^6(\T^d)}
\end{equation*}
for all $u\in H_0^3$.
Hence combining this with Young's inequality, (H$_f^1$),
(H$_f^2$) and (H$_g^3$), for any $0<\veps<1$ we have
\begin{align}\label{H202}
\mathbb E\|u^{n}(t)\|_{2}^{2p}
&
\leq\mathbb E\|\zeta_{s}\|_{2}^{2p}+\mathbb E\int_{s}^{t}
p\|u^{n}(\sigma)\|_{2}^{2p-2}\Big((-2+\veps)\|u^{n}(\sigma)\|_{3}^2
+C_\veps\|u^{n}(\sigma)\|_{1}^{10}\\\nonumber
&\qquad
+C_{\veps}\|u^{n}(\sigma)\|_{1}^2+(2p-1+\veps)L_g^2\|u^{n}(\sigma)\|_{2}^2
+C_\veps\Big)\d\sigma\\\nonumber
&
\leq\mathbb E\|\zeta_{s}\|_{2}^{2p}+\mathbb E\int_{s}^{t}\Big(
p\left((-2+\veps)\lambda_*+(2p-1+\veps)L_g^2+\veps\right)\|u^{n}(\sigma)\|_{2}^{2p}\\\nonumber
&\qquad
+C_\veps\|u^{n}(\sigma)\|_{1}^{10p}+C_\veps\Big)\d\sigma.
\end{align}
Then it follows from Gronwall's lemma and \eqref{uGLH1} that
\begin{align*}%\label{H203}
\sup_{s\leq t\leq s+T}\mathbb E\|u^{n}(t)\|_{2}^{2p}
&
\leq C_T\left(\mathbb E\|\zeta_{s}\|_{2}^{2p}
+\mbE\int_s^{s+T}\left(\|u^{n}(\sigma)\|_{1}^{10p}
+1\right)\d\sigma\right)\\\nonumber
&
\leq C_T\left(\mathbb E\|\zeta_{s}\|_{2}^{2p}
+\mathbb E\|\zeta_{s}\|_{1}^{10p}+1\right),
\end{align*}
which completes the proof of \eqref{uniest3} by
the reflexivity of $\mcL^{2p}(\Omega,\mbP;H_0^2)$.
\end{proof}

Note that $Pr_{2}(L^2(\T^d))$ is a separable complete metric space endowed with
the following Wasserstein distance
\[
W_{2}(\mu_{1},\mu_{2}):=\inf_{\pi\in\mathcal C(\mu_{1},\mu_{2})}
\left(\int_{L^2(\T^d)\times L^2(\T^d)}\|x-y\|^{2}\pi(\d x,\d y)\right)^{\frac{1}{2}}
\]
for all $\mu_{1},\mu_{2}\in Pr_{2}(L^2(\T^d))$. Here
$\mathcal C(\mu_{1},\mu_{2})$ is the set of all couplings for
$\mu_{1}$ and $\mu_{2}$. Now we investigate continuous dependence
on initial values and coefficients for solutions to \eqref{maineq1}.
Recall that a sequence $\{\mu_n\}\subset Pr(L^2(\T^d))$ is said to
{\em weakly converge} to $\mu$ if
$\int f \d\mu_n\to \int f \d\mu$ for all $f\in C_b(L^2(\T^d))$,
and the Wasserstein distance $W_{2}$ metrizes weak convergence.

\begin{lemma}\label{conpth}
Let $\inf\{\gamma_n(t),\gamma(t):n\geq1,t\in\R\}\geq\frac{|\beta|}{\sqrt3}$.
Suppose that $f_{n}$, $f$, $g_{n}$, $g$ satisfy {\rm(H$_f^1$)} and {\rm(H$_g^1$)}
with the same constants. Fix $s\in\R$. Let $u_{n}$ be the solution of
\begin{equation*}
 \left\{
   \begin{aligned}
   &\ \d u_n(t)=\left[(1+i\alpha )\Delta u_n(t)-(\gamma_n(t)+i\beta)|u_n(t)|^{2}u_n(t)
     +f_{n}(t,u_n(t))\right]\d t + g_{n}(t,u_n(t))\d W(t)\\
  &\ u_n(s)=\zeta^{s}_{n}
   \end{aligned}
   \right.
  \end{equation*}
and $u$ the solution of
\begin{equation*}
  \left\{
   \begin{aligned}
   &\ \d u(t)=\left[(1+i\alpha )\Delta u(t)-(\gamma(t)+i\beta)|u(t)|^{2}u(t)
     +f(t,u(t))\right]\d t + g(t,u(t))\d W(t)\\
  &\ u(s)=\zeta^{s}.
   \end{aligned}
   \right.
  \end{equation*}
Assume further that
\begin{enumerate}
\item $\lim\limits_{n\rightarrow\infty}\gamma_{n}(t)=\gamma(t)$
for all $t\in\R$;
\item $\lim\limits_{n\rightarrow\infty}f_{n}(t,x)
=f(t,x)$ for all $t\in\R$ and $x\in L^{2}(\T^{d})$;
\item $\lim\limits_{n\rightarrow\infty}g_{n}(t,x)
=g(t,x)$ for all $t\in\R$ and $x\in L^{2}(\T^{d})$.
\end{enumerate}
If
  $\lim\limits_{n\rightarrow\infty}W_{2}
  (\mathscr{L}(\zeta_{n}^{s}),\mathscr{L}(\zeta^{s}))=0$, then
  for all $T>0$
  $$
  \lim\limits_{n\rightarrow\infty}\sup\limits_{s\leq t\leq s+T}
  W_{2}(\mathscr{L}(u_{n}(t)),\mathscr{L}(u(t)))=0.
  $$
\end{lemma}
\begin{proof}
The proof is similar to that of Theorem 3.1 in \cite{CML2020}.
\end{proof}

\section{The first Bogolyubov theorem}\label{firstBT}
Consider the following stochastic CGL equations with highly
oscillating components
\begin{align}\label{eqLSDE1.1}
\d u^{\veps}(t)
&
=\left[(1+i\alpha )\Delta u^{\veps}(t)
-(\gamma(t/\veps)+i\beta)|u^{\veps}(t)|^{2}u^{\veps}(t)
+f\left(t/\veps,u^{\veps}(t)\right)
\right]\d t\\\nonumber
&\quad
+g\left(t/\veps,u^{\veps}(t)\right)\d W(t),\quad t\in\R,
\end{align}
where $\gamma\in C_b(\R,\R_+)$,
$f\in$ $ C(\mathbb R\times L^{2}(\T^{d}),L^{2}(\T^{d}))$,
$g\in C(\R\times L^{2}(\T^{d}),L_{2}(U,L^{2}(\T^{d})))$ and $0<\veps\ll1$.

We employ $\Psi$ to denote the space of all decreasing, positive
bounded functions $\delta_{1}:\R_{+}\rightarrow\R_{+}$ with
$\lim\limits_{t\rightarrow+\infty}\delta_{1}(t)=0$.
Below we need additional conditions.
\begin{enumerate}
\item[\textbf{(G$_\gamma$)}]
There exists a constant $\bar{\gamma}$ such that
\begin{equation*}
\frac{1}{T}\left|\int_t^{t+T}\left(\gamma(s)-\bar{\gamma}
\right)\d s\right|\leq\delta_\gamma(T)
\end{equation*}
for all $T>0$ and $t\in\R$.
\item[\textbf{(G$_f$)}]
There exist functions $\delta_{f}\in \Psi$ and $\bar{f}\in
C(L^{2}(\T^{d}),L^{2}(\T^{d}))$ such that
\begin{equation}\label{eqG3}
\frac{1}{T}\left\|\int_{t}^{t+T}\left(f(s,x)-\bar{f}(x)\right)\d s\right\|
\le\delta_{f}(T)(1+\|x\|)\nonumber
\end{equation}
for any $T>0$, $t\in\mathbb R$ and $x\in L^2(\T^d)$.
 \item[\textbf{(G$_g^1$)}]
There exist functions $\delta_{g}\in \Psi$ and $\bar{g}\in
C(L^{2}(\T^{d}),L_{2}(U,L^{2}(\T^{d})))$ such that
\begin{equation}\label{eqG4}
\frac{1}{T}\int_{t}^{t+T}\left\|g(s,x)-
\bar{g}(x)\right\|_{L_{2}(U,L^{2}(\T^{d}))}^{2}
\d s\le\delta_{g}(T)(1+\|x\|^2)\nonumber
\end{equation}
for any $T>0$, $t\in\mathbb R$ and $x\in L^2(\T^d)$.
 \item[\textbf{(G$_g^2$)}]
There exist functions $\delta_{g}\in \Psi$ and $\bar{g}\in
C(H_0^1,L_{2}(U,H_0^1))$ such that
\begin{equation}\label{eqG4}
\frac{1}{T}\int_{t}^{t+T}\left\|g(s,x)-
\bar{g}(x)\right\|_{L_{2}(U,H_0^1)}^{2}
\d s\le\delta_{g}(T)(1+\|x\|_1^2)\nonumber
\end{equation}
for any $T>0$, $t\in\mathbb R$ and $x\in H_0^1$.
\end{enumerate}

Set $\gamma_\veps(t):=\gamma(\frac{t}{\veps})$,
$f_{\veps}(t,x):=f(\frac{t}{\veps},x)$ and
$g_{\veps}(t,x):=g(\frac{t}{\veps},x)$
for $t\in\R$, $x\in L^{2}(\T^{d})$ and $\veps \in(0,1]$.
Equation \eqref{eqLSDE1.1} can be written as
\begin{align}\label{eqG2.1}
\d u^{\veps}(t)=
&
\left[(1+i\alpha )\Delta u^{\veps}(t)
-(\gamma_\veps(t)+i\beta)|u^{\veps}(t)|^{2}u^{\veps}(t)
+f_{\veps}(t,u^{\veps}(t))\right]\d t\\\nonumber
&
+g_{\veps}(t,u^{\veps}(t))\d W(t),~ t\in\R.
\end{align}
Along with equations \eqref{eqLSDE1.1}--\eqref{eqG2.1} we
consider the following averaged equation
\begin{equation}\label{eqG5_1}
\d \bar{u}(t)=\left[(1+i\alpha )\Delta \bar{u}(t)-(\bar{\gamma}+i\beta)|\bar{u}(t)|^{2}\bar{u}(t)
+\bar{f}(\bar{u}(t))\right]\d t
+\bar{g}(\bar{u}(t))\d W(t),~ t\in\R.
\end{equation}

\begin{remark}\label{reuniest}\rm
If $f$ and $g$ satisfy {\rm(H$_f^1$)}, {\rm(H$_g^1$)}--{\rm(H$_g^2$)}, {\rm(G$_f^1$)}
and {\rm(G$_g^1$)}--{\rm(G$_g^2$)}, then $\bar{f}$ and $\bar{g}$ also satisfy
{\rm(H$_f^1$)}, {\rm(H$_g^1$)} and {\rm(H$_g^2$)} with the same constants.
Therefore, under the same conditions, \eqref{uniest1} holds uniformly for
$0<\veps\leq1$, $\bar{f}$ and $\bar{g}$.
\end{remark}

For a given process $\varphi$, we define a step process $\widetilde{\varphi}$
such that $\widetilde{\varphi}(\sigma)=\varphi(s+k\delta)$
for any $\sigma\in[s+k\delta,s+(k+1)\delta)$.
Employing the technique of time discretization
(see e.g. \cite{CLAM2021, HLL2021}), we have
the following estimates of solutions for the integral of time
increment on $L^2(\T^d)$ and $L^4(\T^d)$. Note that estimates
on $L^4(\T^d)$ play an important role to
deal with the cubic term involved in averaging.
\begin{lemma}\label{orstes}
Let $\gamma(\cdot)\geq\frac{|\beta|}{\sqrt3}$.
Assume that  {\rm(H$_f^1$)}, {\rm(H$_g^1$)}, {\rm(G$_\gamma$)}, {\rm(G$_f$)} and {\rm(G$_g^1$)} hold.
Let $u^{\veps}(t,s,\zeta_s^\veps)$ be the solution of \eqref{eqG2.1}
with the initial condition
$u^{\veps}(s,s,\zeta_s^\veps)=\zeta_{s}^{\veps}$ and $\bar{u}(t,s,\zeta_s)$ the
solution of \eqref{eqG5_1} with the initial condition
$\bar{u}(s,s,\zeta_s)=\zeta_{s}$. Then we have
\begin{equation}\label{orsteseq}
\E\int_{s}^{s+T}\|u^{\veps}(\sigma,s,\zeta_s^\veps)
-\widetilde{u}^{\veps}(\sigma,s,\zeta_s^\veps)
\|^{2}\d\sigma\leq
C_{T}(1+\E\|\zeta_{s}^{\veps}\|^{2})\delta^{\frac{1}{2}}
\end{equation}
and
\begin{equation}\label{avesteseq}
\E\int_{s}^{s+T}\|\bar{u}(\sigma,s,\zeta_s)-\widetilde{u}(\sigma,s,\zeta_s)\|^{2}\d\sigma
\leq C_{T}(1+\E\|\zeta_{s}\|^{2})\delta^{\frac{1}{2}}
\end{equation}
for any $s\in\R$ and $T>0$, where
$\widetilde{u}(\cdot,s,\zeta_s):=\widetilde{\bar{u}}(\cdot,s,\zeta_s)$.

Moreover, assume that {\rm(H$_g^2$)} and {\rm(G$_g^2$)} hold,
$\zeta_s^\veps,\zeta_s\in \mcL^4(\Omega,\P;H_0^1)$.
Fix $R\in\R_+$. Define
\[
\tau_R^\veps:=\inf\{s\leq t\leq s+T: \|u^\veps(t,s,\zeta_s^\veps)\|+\|\bar{u}(t,s,\zeta_s)\|>R\}.
\]
Then we have
\begin{align}\label{Horsteseq}
\E\int_{s}^{(s+T)\wedge\tau_R^\veps}\|u^{\veps}(\sigma,s,\zeta_s^\veps)
-\widetilde{u}^{\veps}(\sigma,s,\zeta_s^\veps)
\|_{L^4(\T^d)}^{4}\d\sigma\leq R^{\frac12}C_T\left(1+\mbE\|\zeta_s^\veps\|_{1}^4\right)
\delta^{\frac{1}{8}}
\end{align}
and
\begin{align}\label{Hasteseq}
\E\int_{s}^{(s+T)\wedge\tau_R^\veps}\|\bar{u}(\sigma,s,\zeta_s^\veps)
-\widetilde{u}(\sigma,s,\zeta_s^\veps)\|_{L^4(\T^d)}^{4}\d\sigma
\leq R^{\frac12}C_T\left(1+\mbE\|\zeta_s\|_{1}^4\right)
\delta^{\frac{1}{8}}.
\end{align}
\end{lemma}
\begin{proof}
We only need to prove the estimates \eqref{orsteseq} and \eqref{Horsteseq}
when $s=0$; \eqref{avesteseq} and \eqref{Hasteseq} are similar.
Set $u_\veps(\sigma):=u^{\veps}(\sigma,0,\zeta_0^\veps)$
and $T(\delta):=\left[\frac{T}{\delta}\right]$.
It follows from \eqref{uniest1} and Remark \ref{reuniest} that
\begin{align}\label{orsteseq1}
&
\E\int_0^{T}\|u^{\veps}(\sigma)
-\widetilde{u}^{\veps}(\sigma)\|^{2}\d\sigma\\\nonumber
&
=\E\int_0^{\delta}\|u^{\veps}(\sigma)
-\zeta_{0}^{\veps}\|^{2}\d\sigma+\E\sum_{k=1}^{T(\delta)-1}
\int_{k\delta}^{(k+1)\delta}\|u^{\veps}(\sigma)
-u^{\veps}(k\delta)\|^{2}\d\sigma\\\nonumber
&\quad
+\E\int_{T(\delta)\delta}^{T}\|u^{\veps}(\sigma)
-u^{\veps}(T(\delta)\delta)\|^{2}\d\sigma\\\nonumber
&
\leq C_{T}\left(1+\E\|\zeta_{0}^{\veps}\|^{2}\right)\delta
+2\E\sum_{k=1}^{T(\delta)-1}\int_{k\delta}^{(k+1)\delta}
\|u^{\veps}(\sigma)-u^{\veps}(\sigma-\delta)\|^{2}
\d\sigma\\\nonumber
&\quad
+2\E\sum_{k=1}^{T(\delta)-1}\int_{k\delta}^{(k+1)\delta}
\|u^{\veps}(\sigma-\delta)-u^{\veps}(k\delta)\|^{2}
\d\sigma\\\nonumber
&
=:C_{T}\left(1+\E\|\zeta_{0}^{\veps}\|^{2}\right)\delta
+2\sum_{k=1}^{T(\delta)-1}\mathcal I_{k}
+2\sum_{k=1}^{T(\delta)-1}\mathcal J_{k}.
\end{align}

Given $1\leq k\leq T(\delta)-1$, for any
$\sigma\in[k\delta,(k+1)\delta)$, by It\^o's formula,
we get
\begin{align*}
&
\|u^{\veps}(\sigma)-u^{\veps}(\sigma-\delta)\|^{2}\\
&
=2\int_{\sigma-\delta}^{\sigma}\langle (1+i\alpha)\Delta u^{\veps}(\tau)
-(\gamma_\veps(\tau)+i\beta)|u^{\veps}(\tau)|^{2}u^{\veps}(\tau),
u^{\veps}(\tau)-u^{\veps}(\sigma-\delta)\rangle
\d\tau\\
&\quad
+2\int_{\sigma-\delta}^{\sigma}\langle
f_{\veps}(\tau,u^{\veps}(\tau)),
u^{\veps}(\tau)-u^{\veps}(\sigma-\delta)\rangle\d\tau
+\int_{\sigma-\delta}^{\sigma}\|g_{\veps}(\tau,
u^{\veps}(\tau))\|_{L_{2}(U,L^{2}(\T^{d}))}^{2}
\d\tau\\
&\quad
+2\int_{\sigma-\delta}^{\sigma}\langle u^{\veps}(\tau)
-u^{\veps}(\sigma-\delta),
g_{\veps}(\tau,u^{\veps}(\tau))
\d W(\tau)\rangle.
\end{align*}
Hence it follows from (H$_f^1$), (H$_g^1$), integration by parts and
Young's inequality that
\begin{align*}
&
\|u^{\veps}(\sigma)-u^{\veps}(\sigma-\delta)\|^{2}\\
&
\leq\int_{\sigma-\delta}^{\sigma}\bigg(
2\langle(1+i\alpha)\nabla u^{\veps}(\tau),
\nabla u^{\veps}(\sigma-\delta)\rangle
+2\langle(\gamma_\veps(\tau)+i\beta)|u^{\veps}(\tau)|^{2}u^{\veps}(\tau),
u^{\veps}(\sigma-\delta)\rangle\\
&\qquad
+2\langle f_{\veps}(\tau,u^{\veps}(\tau)),
u^{\veps}(\tau)-u^{\veps}(\sigma-\delta)\rangle
+2L_{g}^{2}\|u^{\veps}(\tau)\|^{2}+2K^{2}\bigg)\d\tau\\
&\quad
+2\int_{\sigma-\delta}^{\sigma}\langle u^{\veps}(\tau)
-u^{\veps}(\sigma-\delta),
g_{\veps}(\tau,u^{\veps}(\tau))
\d W(\tau)\rangle\\
&
\leq\int_{\sigma-\delta}^{\sigma}C\bigg(\|u^{\veps}(\tau)\|_{1}
\|u^{\veps}(\sigma-\delta)\|_{1}
+\|u^{\veps}(\tau)\|_{L^{4}(\T^{d})}^{3}
\|u^{\veps}(\sigma-\delta)\|_{L^{4}(\T^{d})}\\
&\qquad
+\|f_{\veps}(\tau,u^{\veps}(\tau))\|
\|u^{\veps}(\tau)-u^{\veps}(\sigma-\delta)\|
+\|u^{\veps}(\tau)\|^{2}+1\bigg)\d\tau\\
&\quad
+2\int_{\sigma-\delta}^{\sigma}\langle u^{\veps}(\tau)
-u^{\veps}(\sigma-\delta),g_{\veps}(\tau,u^{\veps}(\tau))
\d W(\tau)\rangle\\
&
\leq\int_{\sigma-\delta}^{\sigma}C\bigg(
\|u^{\veps}(\tau)\|_{1}^{2}
+\|u^{\veps}(\sigma-\delta)\|_{1}^{2}
+\|u^{\veps}(\tau)\|_{L^{4}(\T^{d})}^{4}\\
&\qquad
+\|u^{\veps}(\sigma-\delta)\|_{L^{4}(\T^{d})}^{4}
+\|u^{\veps}(\tau)\|^{2}
+\|u^{\veps}(\sigma-\delta)\|^{2}+1\bigg)\d\tau\\
&\quad
+2\int_{\sigma-\delta}^{\sigma}\langle u^{\veps}(\tau)
-u^{\veps}(\sigma-\delta),g_{\veps}(\tau,u^{\veps}(\tau))
\d W(\tau)\rangle.\\
\end{align*}
Then we have
\begin{align}\label{orsteseq3}
\mathcal I_{k}
&
:=\E\int_{k\delta}^{(k+1)\delta}\|u^{\veps}(\sigma)
-u^{\veps}(\sigma-\delta)\|^{2}\d\sigma\\\nonumber
&
\leq \E\int_{k\delta}^{(k+1)\delta}\bigg
\{\int_{\sigma-\delta}^{\sigma}C\bigg(
\|u^{\veps}(\tau)\|_{1}^{2}
+\|u^{\veps}(\sigma-\delta)\|_{1}^{2}
+\|u^{\veps}(\tau)\|_{L^{4}(\T^{d})}^{4}\\\nonumber
&\qquad
+\|u^{\veps}(\sigma-\delta)\|_{L^{4}(\T^{d})}^{4}
+\|u^{\veps}(\tau)\|^{2}
+\|u^{\veps}(\sigma-\delta)\|^{2}+1\bigg)\d\tau\\\nonumber
&\qquad
+2\int_{\sigma-\delta}^{\sigma}\langle u^{\veps}(\tau)
-u^{\veps}(\sigma-\delta),g_{\veps}(\tau,u^{\veps}(\tau))
\d W(\tau)\rangle\bigg\}\d\sigma
=:\mathcal I_{k}^{1}+\mathcal I_{k}^{2}.
\end{align}
For $\mathcal I_{k}^{1}$, by Fubini's theorem and a change of variable, we have
\begin{align}\label{orsteseq3.1}
\mathcal I_{k}^{1}
&
:=\E\int_{k\delta}^{(k+1)\delta}\bigg\{
\int_{\sigma-\delta}^{\sigma}C\bigg(
\|u^{\veps}(\tau)\|_{1}^{2}
+\|u^{\veps}(\sigma-\delta)\|_{1}^{2}
+\|u^{\veps}(\tau)\|_{L^{4}(\T^{d})}^{4}\\\nonumber
&\qquad
+\|u^{\veps}(\sigma-\delta)\|_{L^{4}(\T^{d})}^{4}
+\|u^{\veps}(\tau)\|^{2}
+\|u^{\veps}(\sigma-\delta)\|^{2}+1\bigg)\d\tau\bigg\}
\d\sigma\\\nonumber
&
=\E\int_{k\delta}^{(k+1)\delta}\int_{\sigma-\delta}^{\sigma}
C\left(\|u^{\veps}(\tau)\|_{1}^{2}
+\|u^{\veps}(\tau)\|_{L^{4}(\T^{d})}^{4}
+\|u^{\veps}(\tau)\|^{2}+1\right)\d\tau\d\sigma\\\nonumber
&\quad
+\E\int_{k\delta}^{(k+1)\delta}\delta
C\left(\|u^{\veps}(\sigma-\delta)\|_{1}^{2}
+\|u^{\veps}(\sigma-\delta)\|_{L^{4}(\T^{d})}^{4}
+\|u^{\veps}(\sigma-\delta)\|^{2}\right)\d\sigma\\\nonumber
&
\leq\delta C\mathbb E\int_{(k-1)\delta}^{(k+1)\delta}\left(
\|u^{\veps}(\tau)\|_{1}^{2}
+\|u^{\veps}(\tau)\|_{L^{4}(\T^{d})}^{4}
+\|u^{\veps}(\tau)\|^{2}+1\right)\d\tau.
\end{align}

Now we estimate $\mathcal I_{k}^{2}$. In view of Burkholder-Davis-Gundy
inequality, (H$_g^1$) and Young's inequality, we obtain
\begin{align*}
\mathcal I_{k}^{2}
&
:=2\E\int_{k\delta}^{(k+1)\delta}\int_{\sigma-\delta}^{\sigma}
\langle u^{\veps}(\tau)-u^{\veps}(\sigma-\delta),
g_{\veps}(\tau,u^{\veps}(\tau))
\d W(\tau)\rangle\d\sigma\\\nonumber
&
\leq C\int_{k\delta}^{(k+1)\delta}\E\left(\int_{\sigma-\delta}^{\sigma}
\|g_{\veps}(\tau,u^{\veps}(\tau))
\|_{L_{2}(U,L^{2}(\T^{d}))}^{2}
\|u^{\veps}(\tau)-u^{\veps}(\sigma-\delta)\|^{2}\d\tau
\right)^{\frac{1}{2}}\d\sigma\\\nonumber
&
\leq C\int_{k\delta}^{(k+1)\delta}\E\left(\int_{\sigma-\delta}^{\sigma}
\left(\|u^{\veps}(\tau)\|^{2}+1\right)
\|u^{\veps}(\tau)-u^{\veps}(\sigma-\delta)\|^{2}\d\tau
\right)^{\frac{1}{2}}\d\sigma\\\nonumber
&
\leq\delta^{\frac{1}{2}}C\left[\int_{k\delta}^{(k+1)\delta}
\E\int_{\sigma-\delta}^{\sigma}\left(\|u^{\veps}(\tau)\|^{4}
+\|u^{\veps}(\sigma-\delta)\|^{4}+1\right)\d\tau\d\sigma
\right]^{\frac{1}{2}},
\end{align*}
which implies that
\begin{align}\label{orsteseq4}
\mathcal I_{k}^{2}
&
\leq\delta^{\frac{1}{2}}C\left(\E\int_{(k-1)\delta}^{(k+1)\delta}
\delta\|u^{\veps}(\tau)\|^{4}\d\tau
+\delta^{2}\right)^{\frac{1}{2}}\\\nonumber
&
\leq\delta C\left(\E\int_{(k-1)\delta}^{(k+1)\delta}
\|u^{\veps}(\tau)\|_{L^{4}(\T^{d})}^{4}\d\tau
\right)^{\frac{1}{2}}+C\delta^{\frac{3}{2}}
\end{align}
by Fubini's theorem. Therefore \eqref{orsteseq3}--\eqref{orsteseq4} yield
\begin{align*}
\mathcal I_{k}
&
\leq \delta C
\E\int_{(k-1)\delta}^{(k+1)\delta}
\left(\|u^{\veps}(\tau)\|^{2}_{1}
+\|u^{\veps}(\tau)\|^{4}_{L^{4}(\T^{d})}
+\|u^{\veps}(\tau)\|^{2}+1\right)\d\tau\\
&\quad
+\delta C\left(\E\int_{(k-1)\delta}^{(k+1)\delta}
\|u^{\veps}(\tau)\|_{L^{4}(\T^{d})}^{4}\d\tau
\right)^{\frac{1}{2}}+C\delta^{\frac{3}{2}}.
\end{align*}
By Lemma \ref{uniest} and Remark \ref{reuniest}, we get
\begin{align}\label{orsteseq6}
2\sum_{k=1}^{T(\delta)-1}\mathcal I_{k}
&
\leq\delta C \E\int_0^{T}
\left(\|u^{\veps}(\tau)\|^{2}_{1}
+\|u^{\veps}(\tau)\|^{4}_{L^{4}(\T^{d})}
+\|u^{\veps}(\tau)\|^{2}+1\right)\d\tau\\\nonumber
&\quad
+\delta C\sum_{k=1}^{T(\delta)-1}\left(\E
\int_{(k-1)\delta}^{(k+1)\delta}
\|u^{\veps}(\tau)\|_{L^{4}(\T^{d})}^{4}\d\tau
\right)^{\frac{1}{2}}+C_{T}\delta^{\frac{1}{2}}\\\nonumber
&
\leq C_{T}\left(1+\E\|\zeta_{0}^{\veps}\|^{2}\right)
\delta^{\frac{1}{2}}
+\delta C\left(T(\delta)\right)^{\frac{1}{2}}\left(
\sum_{k=1}^{T(\delta)-1}\E\int_{(k-1)\delta}^{(k+1)\delta}
\|u^{\veps}(\tau)\|_{L^{4}(\T^{d})}^{4}\d\tau
\right)^{\frac{1}{2}}\\\nonumber
&
\leq C_{T}\left(1+\E\|\zeta_{0}^{\veps}\|^{2}\right)
\delta^{\frac{1}{2}}.
\end{align}

Similarly, we have
\begin{align}\label{orsteseq7}
2\sum_{k=1}^{T(\delta)-1}\mathcal J_{k}
&
\leq C_{T}\left(1+\E\|\zeta_{0}^{\veps}\|^{2}\right)
\delta^{\frac{1}{2}}.
\end{align}
Combining \eqref{orsteseq1}, \eqref{orsteseq6} and \eqref{orsteseq7},
we obtain
\begin{equation*}
\E\int_0^{T}\|u^{\veps}(\sigma)-\widetilde{u}^{\veps}
(\sigma)\|^{2}\d\sigma\leq C_{T}(1
+\E\|\zeta_{0}^{\veps}\|^{2})\delta^{\frac{1}{2}}.
\end{equation*}

Assume that {\rm(H$_g^2$)} and {\rm(G$_g^2$)} hold, and
$\zeta_s^\veps,\zeta_s\in \mcL^4(\Omega,\P;H_0^1)$.
Note that
\begin{equation}\label{intp}
\|v\|_{L^4(\T^d)}^4\leq C\|v\|\|v\|_{1}^3
\end{equation}
for all $v\in H_0^1$.
In view of \eqref{intp}, \eqref{uniest2} and \eqref{orsteseq}, we have
\begin{align*}%\label{H01}
&
\E\int_{0}^{T\wedge\tau_R^\veps}\|u^{\veps}(\sigma)
-\widetilde{u}^{\veps}(\sigma)
\|_{L^4(\T^d)}^{4}\d\sigma\\\nonumber
&
\leq C\E\int_{0}^{T\wedge\tau_R^\veps}\|u^{\veps}(\sigma)
-\widetilde{u}^{\veps}(\sigma)\|
\|u^{\veps}(\sigma)-\widetilde{u}^{\veps}(\sigma)
\|_{1}^{3}\d\sigma\\\nonumber
&
\leq C\left(\E\int_{0}^{T\wedge\tau_R^\veps}\|u^{\veps}(\sigma)
-\widetilde{u}^{\veps}(\sigma)
\|^4\d\sigma\right)^{\frac14}
\left(\E\int_{0}^{T}\|u^{\veps}(\sigma)
-\widetilde{u}^{\veps}(\sigma)
\|_{1}^{4}\d\sigma\right)^{\frac34}\\\nonumber
&
\leq C_{T}(1+\E\|\zeta_{0}^{\veps}\|_{1}^4)^{\frac34}
\left(\E\int_{0}^{T\wedge\tau_R^\veps}\|u^{\veps}(\sigma)
-\widetilde{u}^{\veps}(\sigma)\|^4\d\sigma\right)^{\frac14}\\\nonumber
&
\leq R^{\frac12}C_{T}(1+\E\|\zeta_{0}^{\veps}\|_{1}^4)\delta^{\frac18}.
\end{align*}

\end{proof}

\begin{remark}\rm
Under conditions of Lemma \ref{orstes}, by \eqref{Horsteseq}, \eqref{Hasteseq} and \eqref{uniest1}
we have
\begin{equation}\label{esteqsp1}
\mbE\int_s^{(s+T)\wedge\tau_R^\veps}\|\widetilde{u}(\sigma)\|_{L^4(\T^d)}^4\d\sigma
\leq C_T\left(1+\mbE\|\zeta_s\|^4_{1}\right)
\end{equation}
and
\begin{equation}\label{esteqsp2}
\mbE\int_s^{(s+T)\wedge\tau_R^\veps}\|\widetilde{u}^\veps(\sigma)\|_{L^4(\T^d)}^4\d\sigma
\leq C_T\left(1+\mbE\|\zeta_s^\veps\|^4_{1}\right).
\end{equation}
\end{remark}

Now we establish the first Bogolyubov theorem for stochastic CGL equations.
\begin{theorem}\label{avethf}
Suppose that {\rm{(H$_f^1$)}}, {\rm(H$_g^1$)}, {\rm(H$_g^2$)},
{\rm(G$_\gamma$)}, {\rm(G$_f$)}, {\rm{(G$_g^1$)}} and {\rm{(G$_g^2$)}} hold. For any
$s\in\R$, let $u^{\veps}(t,s,\zeta_s^\veps),t\geq s$ be the solution of \eqref{eqG2.1}
with the initial condition $u^{\veps}(s,s,\zeta_s^\veps)=\zeta_s^\veps$
and $\bar{u}(t,s,\zeta_s),t\geq s$ the solution of \eqref{eqG5_1}
with the initial condition $\bar{u}(s,s,\zeta_s)=\zeta_s$.
Assume further that
\begin{equation}\label{conkappa}
\kappa:=\sup_{\veps}\mbE\|\zeta_s^\veps\|_{1}^4+
\mbE\|\zeta_s\|_{1}^4<\infty,
\end{equation}
$\gamma_\veps(t)>\frac{|\beta|}{\sqrt{3}}$ for all $t\in\R$
and
$\lim\limits_{\veps\rightarrow0}
\mathbb E\|\zeta^{\veps}_{s}-\zeta_{s}\|^{2}=0$.
Then
$$
\lim_{\veps\rightarrow0} \mathbb E\left(\sup_{s\leq t\leq s+T}
\|u^{\veps}(t,s,\zeta_s^\veps)-\bar{u}(t,s,\zeta_s)\|^{2}\right)=0
$$
for any $T>0$.
\end{theorem}

\begin{proof}
Let $u^\veps(t):=u^{\veps}(t,s,\zeta_s^\veps)$ and
$\bar{u}(t):=\bar{u}(t,s,\zeta_s)$.
Without loss of generality, we assume that $s=0$.
Set $\zeta^\veps:=\zeta_0^\veps$ and $\zeta:=\zeta_0$ for simplicity.
Let $\chi_{A}$ be the indicator function. It follows from Chebyshev's inequality
and \eqref{uniest1} that
\begin{align}\label{stop0}
&
\mathbb E\left(\sup_{0\leq t\leq T}
\|u^{\veps}(t)-\bar{u}(t)\|^{2}\right)\\\nonumber
&
\leq
\mathbb E\left(\sup_{0\leq t\leq T\wedge\tau_R^\veps}
\|u^{\veps}(t)-\bar{u}(t)\|^{2}\right)
+\mathbb E\left(\chi_{\{\tau_R^\veps\leq T\}}\cdot\sup_{0\leq t\leq T}
\|u^{\veps}(t)-\bar{u}(t)\|^{2}\right)\\\nonumber
&
\leq \mathbb E\left(\sup_{0\leq t\leq T\wedge\tau_R^\veps}
\|u^{\veps}(t)-\bar{u}(t)\|^{2}\right)
+R^{-2}C
\left(\mbE\left(\sup_{0\leq t\leq T}\|u^\veps(t)\|^4\right)
+\mbE\left(\sup_{0\leq t\leq T}\|\bar{u}(t)\|^4\right)\right)\\\nonumber
&
\leq \mathbb E\left(\sup_{0\leq t\leq T\wedge\tau_R^\veps}
\|u^{\veps}(t)-\bar{u}(t)\|^{2}\right)
+R^{-2}C_{T,\kappa}.
\end{align}
In view of It\^o's formula and integration by parts, we have
\begin{align*}
&
\|u^{\veps}(t)-\bar{u}(t)\|^{2}\\\nonumber
&
=\|\zeta^{\veps}-\zeta\|^{2}+\int_{0}^{t}\bigg(2\langle
(1+i\alpha)\Delta(u^{\veps}(\sigma)-\bar{u}(\sigma)),
u^{\veps}(\sigma)-\bar{u}(\sigma)\rangle\\\nonumber
&\qquad
-2\langle(\gamma_\veps(\sigma)+i\beta)|u^{\veps}(\sigma)|^{2}
u^{\veps}(\sigma)-
(\bar{\gamma}+i\beta)|\bar{u}(\sigma)|^{2}\bar{u}(\sigma),
u^{\veps}(\sigma)-\bar{u}(\sigma)\rangle\\\nonumber
&\qquad
+2\langle f_{\veps}(\sigma,u^{\veps}(\sigma))
-\bar{f}(\bar{u}(\sigma)),u^{\veps}(\sigma)-\bar{u}(\sigma)\rangle
+\|g_{\veps}(\sigma,u^{\veps}(\sigma))
-\bar{g}(\bar{u}(\sigma))\|_{L_{2}(U,L^{2}(\T^{d}))}^{2}\bigg)\d\sigma\\\nonumber
&\quad
+2\int_{0}^{t}\langle u^{\veps}(\sigma)-\bar{u}(\sigma),\left(
g_{\veps}(\sigma,u^{\veps}(\sigma))
-\bar{g}(\bar{u}(\sigma))\right)\d W(\sigma)\rangle\\\nonumber
&
\leq\|\zeta^{\veps}-\zeta\|^{2}+\int_{0}^{t}\bigg(
-2\langle(\gamma_\veps(\sigma)+i\beta)|u^{\veps}(\sigma)|^{2}
u^{\veps}(\sigma)-
(\bar{\gamma}+i\beta)|\bar{u}(\sigma)|^{2}\bar{u}(\sigma),
u^{\veps}(\sigma)-\bar{u}(\sigma)\rangle\\\nonumber
&\qquad
+2\langle f_{\veps}(\sigma,u^{\veps}(\sigma))
-\bar{f}(\bar{u}(\sigma)),u^{\veps}(\sigma)-\bar{u}(\sigma)\rangle
+\|g_{\veps}(\sigma,u^{\veps}(\sigma))
-\bar{g}(\bar{u}(\sigma))\|_{L_{2}(U,L^{2}(\T^{d}))}^{2}\bigg)\d\sigma\\\nonumber
&\quad
+2\int_{0}^{t}\langle u^{\veps}(\sigma)-\bar{u}(\sigma),\left(
g_{\veps}(\sigma,u^{\veps}(\sigma))
-\bar{g}(\bar{u}(\sigma))\right)\d W(\sigma)\rangle.
\end{align*}

Then by Burkholder-Davis-Gundy inequality we get
\begin{align*}
&
\mathbb E\left(\sup_{0\leq t\leq T\wedge\tau_R^\veps}
\|u^{\veps}(t)-\bar{u}(t)\|^{2}\right)\\\nonumber
&
\leq \mathbb E\|\zeta^{\veps}-\zeta\|^{2}
+\mathbb E\int_{0}^{T}
\|g_{\veps}(\sigma,u^{\veps}(\sigma))
-\bar{g}(\bar{u}(\sigma))\|_{L_{2}(U,L^{2}(\T^{d}))}^{2}
\d\sigma\\\nonumber
&\quad
+\mathbb E\left(\sup_{0\leq t\leq T\wedge\tau_R^\veps}\int_{0}^{t}
-2\langle(\gamma_\veps(\sigma)+i\beta)|u^{\veps}(\sigma)|^{2}
u^{\veps}(\sigma)-(\bar{\gamma}+i\beta)|\bar{u}(\sigma)|^{2}\bar{u}(\sigma),
u^{\veps}(\sigma)-\bar{u}(\sigma)\rangle\d\sigma\right)\\\nonumber
&\quad
+\mathbb E\left(\sup_{0\leq t\leq T}\int_{0}^{t}2\langle
f_{\veps}(\sigma,u^{\veps}(\sigma))
-\bar{f}(\bar{u}(\sigma)),u^{\veps}(\sigma)-\bar{u}(\sigma)
\rangle\d\sigma\right)\\\nonumber
&\quad
+6\mathbb E\left(\int_{0}^{T}
\|g_{\veps}(\sigma,u^{\veps}(\sigma))
-\bar{g}(\bar{u}(\sigma))\|_{L_{2}(U,L^{2}(\T^{d}))}^{2}
\|u^{\veps}(\sigma)-\bar{u}(\sigma)\|^{2}\d\sigma
\right)^{\frac{1}{2}}.
\end{align*}
We note that
\begin{align*}
&
6\mathbb E\left(\int_{0}^{T}
\|g_{\veps}(\sigma,u^{\veps}(\sigma))
-\bar{g}(\bar{u}(\sigma))\|_{L_{2}(U,L^{2}(\T^{d}))}^{2}
\|u^{\veps}(\sigma)-\bar{u}(\sigma)\|^{2}\d\sigma
\right)^{\frac{1}{2}}\\\nonumber
&
\leq\frac{1}{2}\mbE\left(\sup_{0\leq t\leq T}
\|u^{\veps}(t)-\bar{u}(t)\|^{2}\right)+C\mathbb E\int_{0}^{T}\|
g_{\veps}(\sigma,u^{\veps}(\sigma))
-\bar{g}(\bar{u}(\sigma))\|_{L_{2}(U,L^{2}(\T^{d}))}^{2}\d\sigma
\end{align*}
by Young's inequality. Hence
\begin{align}\label{avethfeq3}
\mathbb E\left(\sup_{0\leq t\leq T\wedge\tau_R^\veps}
\|u^{\veps}(t)-\bar{u}(t)\|^{2}\right)\leq
2\mathbb E\|\zeta^{\veps}-\zeta\|^{2}+\mathscr I_1
+\mathscr I_{2}+\mathscr I_{3},
\end{align}
where
\begin{align*}
\mathscr I_1:=\mathbb E\left(\sup_{0\leq t\leq T\wedge\tau_R^\veps}\int_{0}^{t}
-4\langle(\gamma_\veps(\sigma)+i\beta)|u^{\veps}(\sigma)|^{2}
u^{\veps}(\sigma)
-(\bar{\gamma}+i\beta)|\bar{u}(\sigma)|^{2}\bar{u}(\sigma),
u^{\veps}(\sigma)-\bar{u}(\sigma)\rangle\d\sigma\right),
\end{align*}
\begin{align*}
\mathscr I_2:=4\mathbb E\left(\sup_{0\leq t\leq T}\int_{0}^{t}\langle
f_{\veps}(\sigma,u^{\veps}(\sigma))
-\bar{f}(\bar{u}(\sigma)),u^{\veps}(\sigma)-\bar{u}(\sigma)
\rangle\d\sigma\right)
\end{align*}
and
\begin{align*}
\mathscr I_3:=C\mathbb E\int_{0}^{T}
\|g_{\veps}(\sigma,u^{\veps}(\sigma))
-\bar{g}(\bar{u}(\sigma))\|_{L_{2}(U,L^{2}(\T^{d}))}^{2}
\d\sigma.
\end{align*}

Firstly, we estimate $\mathscr I_1$.
It follows from Lemma \ref{Disscubic} that
\[
\left\langle-\left(|\beta|/\sqrt{3}+i\beta\right)\left(|u|^{2}u-|v|^{2}v\right),u-v
\right\rangle\leq0
\]
for all $u,v\in H^1(\T^{d})$.
We set $\gamma_\veps(\cdot):=\gamma_\veps(\cdot)-\frac{|\beta|}{\sqrt{3}}\geq0$
and $\bar{\gamma}:=\bar{\gamma}-\frac{|\beta|}{\sqrt{3}}\geq0$. Hence
\begin{align*}%\label{cubic01}
\mathscr I_1
&
\leq\mathbb E\left(\sup_{0\leq t\leq T\wedge\tau_R^\veps}\int_{0}^{t}
-4\langle\gamma_\veps(\sigma)|u^{\veps}(\sigma)|^{2}
u^{\veps}(\sigma)
-\bar{\gamma}|\bar{u}(\sigma)|^{2}\bar{u}(\sigma),
u^{\veps}(\sigma)-\bar{u}(\sigma)\rangle\d\sigma\right)\\\nonumber
&
\leq\mathbb E\left(\sup_{0\leq t\leq T\wedge\tau_R^\veps}\int_{0}^{t}
-4\langle\gamma_\veps(\sigma)|u^{\veps}(\sigma)|^{2}
u^{\veps}(\sigma)
-\gamma_\veps(\sigma)|\bar{u}(\sigma)|^{2}\bar{u}(\sigma),
u^{\veps}(\sigma)-\bar{u}(\sigma)\rangle\d\sigma\right)\\\nonumber
&\quad
+\mathbb E\left(\sup_{0\leq t\leq T\wedge\tau_R^\veps}\int_{0}^{t}
-4\langle\gamma_\veps(\sigma)|\bar{u}(\sigma)|^{2}\bar{u}(\sigma)
-\bar{\gamma}|\bar{u}(\sigma)|^{2}\bar{u}(\sigma),
u^{\veps}(\sigma)-\bar{u}(\sigma)\rangle\d\sigma\right)\\\nonumber
&
\leq \mathbb E\left(\sup_{0\leq t\leq T\wedge\tau_R^\veps}\int_{0}^{t}
-4\langle\gamma_\veps(\sigma)|\bar{u}(\sigma)|^{2}\bar{u}(\sigma)
-\bar{\gamma}|\bar{u}(\sigma)|^{2}\bar{u}(\sigma),
u^{\veps}(\sigma)-\bar{u}(\sigma)\rangle\d\sigma\right)
\end{align*}
since $\langle|u|^{2}u-|v|^{2}v,u-v\rangle\geq0$ for
all $u,v\in H^1$. Therefore, we have
\begin{align}\label{cubic03}
\mathscr I_1
&
\leq \mathbb E\left(\sup_{0\leq t\leq T\wedge\tau_R^\veps}\int_{0}^{t}
-4\langle\gamma_\veps(\sigma)|\bar{u}(\sigma)|^{2}\bar{u}(\sigma)
-\bar{\gamma}|\bar{u}(\sigma)|^{2}\bar{u}(\sigma),
u^{\veps}(\sigma)-\widetilde{u}^{\veps}(\sigma)\rangle\d\sigma\right)\\\nonumber
&\quad
+\mathbb E\left(\sup_{0\leq t\leq T\wedge\tau_R^\veps}\int_{0}^{t}
-4\langle\gamma_\veps(\sigma)|\bar{u}(\sigma)|^{2}\bar{u}(\sigma)
-\bar{\gamma}|\bar{u}(\sigma)|^{2}\bar{u}(\sigma),
\widetilde{u}^{\veps}(\sigma)-\widetilde{u}(\sigma)\rangle\d\sigma\right)\\\nonumber
&\quad
+\mathbb E\left(\sup_{0\leq t\leq T\wedge\tau_R^\veps}\int_{0}^{t}
-4\langle\gamma_\veps(\sigma)|\bar{u}(\sigma)|^{2}\bar{u}(\sigma)
-\bar{\gamma}|\bar{u}(\sigma)|^{2}\bar{u}(\sigma),
\widetilde{u}(\sigma)-\bar{u}(\sigma)\rangle\d\sigma\right)\\\nonumber
&
=:\mathscr I_1^{1}+\mathscr I_1^{2}+\mathscr I_1^{3}.
\end{align}
For $\mathscr I_1^{1}$, by H\"older's inequality, \eqref{uniest1} and \eqref{Horsteseq}, one sees that
\begin{align}\label{cubic04}
\mathscr I_1^{1}
&
\leq C\mathbb E\int_{0}^{T\wedge\tau_R^\veps}\|
\gamma_\veps(\sigma)|\bar{u}(\sigma)|^{2}\bar{u}(\sigma)
-\bar{\gamma}|\bar{u}(\sigma)|^{2}\bar{u}(\sigma)\|_{L^{\frac43}}
\|u^{\veps}(\sigma)-\widetilde{u}^{\veps}(\sigma)\|_{L^4}\d\sigma\\\nonumber
&
\leq C\mathbb E\int_{0}^{T\wedge\tau_R^\veps}\|\bar{u}(\sigma)\|_{L^4}^3
\|u^{\veps}(\sigma)-\widetilde{u}^{\veps}(\sigma)\|_{L^4}\d\sigma\\\nonumber
&
\leq C\left(\mathbb E\int_{0}^{T}\|\bar{u}(\sigma)\|_{L^4}^4\d\sigma\right)^{\frac34}
\left(\mathbb E\int_{0}^{T\wedge\tau_R^\veps}\|u^{\veps}(\sigma)-\widetilde{u}^{\veps}(\sigma)\|_{L^4}^4\d\sigma
\right)^{\frac14}\\\nonumber
&
\leq R^{\frac18}C_{T,\kappa}\delta^{\frac{1}{32}}.
\end{align}
Similarly, by H\"older's inequality, \eqref{uniest1} and \eqref{Hasteseq} we have
\begin{align}\label{cubic05}
\mathscr I_1^{3}
&
\leq C\left(\mathbb E\int_{0}^{T}\|\bar{u}(\sigma)\|_{L^4}^4\d\sigma\right)^{\frac34}
\left(\mathbb E\int_{0}^{T\wedge\tau_R^\veps}\|\bar{u}(\sigma)-\widetilde{u}(\sigma)\|_{L^4}^4\d\sigma
\right)^{\frac14}
\leq R^{\frac18}C_{T,\kappa}\delta^{\frac{1}{32}}.
\end{align}
For $\mathscr I_1^{2}$, we obtain
\begin{align}\label{cubic06}
\mathscr I_1^{2}
&
:=\mathbb E\left(\sup_{0\leq t\leq T\wedge\tau_R^\veps}\int_{0}^{t}
-4\langle\gamma_\veps(\sigma)|\bar{u}(\sigma)|^{2}\bar{u}(\sigma)
-\bar{\gamma}|\bar{u}(\sigma)|^{2}\bar{u}(\sigma),
\widetilde{u}^{\veps}(\sigma)-\widetilde{u}(\sigma)\rangle
\d\sigma\right)\\\nonumber
&
\leq\mathbb E\left(\sup_{0\leq t\leq T\wedge\tau_R^\veps}\int_{0}^{t}
-4\langle\gamma_\veps(\sigma)|\bar{u}(\sigma)|^{2}\bar{u}(\sigma)
-\gamma_\veps(\sigma)|\widetilde{u}(\sigma)|^{2}\widetilde{u}(\sigma),
\widetilde{u}^{\veps}(\sigma)-\widetilde{u}(\sigma)\rangle
\d\sigma\right)\\\nonumber
&\quad
+\mathbb E\left(\sup_{0\leq t\leq T\wedge\tau_R^\veps}\int_{0}^{t}
-4\langle\gamma_\veps(\sigma)|\widetilde{u}(\sigma)|^{2}
\widetilde{u}(\sigma)
-\bar{\gamma}|\widetilde{u}(\sigma)|^{2}\widetilde{u}(\sigma),
\widetilde{u}^{\veps}(\sigma)-\widetilde{u}(\sigma)\rangle
\d\sigma\right)\\\nonumber
&\quad
+\mathbb E\left(\sup_{0\leq t\leq T\wedge\tau_R^\veps}\int_{0}^{t}
-4\langle\bar{\gamma}|\widetilde{u}(\sigma)|^{2}\widetilde{u}(\sigma)
-\bar{\gamma}|\bar{u}(\sigma)|^{2}\bar{u}(\sigma),
\widetilde{u}^{\veps}(\sigma)-\widetilde{u}(\sigma)\rangle
\d\sigma\right)\\\nonumber
&
=:\mathscr I_1^{2,1}+\mathscr I_1^{2,2}+\mathscr I_1^{2,3}.
\end{align}

Note that
\begin{align*}
\left|\la |u|^2u-|v|^2v,w\ra\right|\leq C\|u-v\|_{L^4(\T^d)}
\left(\|u\|_{L^4(\T^d)}^3+\|v\|_{L^4(\T^d)}^3
+\|w\|_{L^4(\T^d)}^3\right)
\end{align*}
for all $u,v,w\in H_0^1$. Therefore,
\begin{align}\label{cubic07}
&
\mathscr I_1^{2,1}+\mathscr I_1^{2,3}\\\nonumber
&
\leq C\mbE\int_{0}^{T\wedge\tau_R^\veps}
\|\widetilde{u}(\sigma)-\bar{u}(\sigma)\|_{L^4(\T^d)}
\left(\|\widetilde{u}(\sigma)\|_{L^4(\T^d)}^3
+\|\bar{u}(\sigma)\|_{L^4(\T^d)}^3
+\|\widetilde{u}^\veps(\sigma)\|_{L^4(\T^d)}^3\right)
\d\sigma\\\nonumber
&
\leq C\left(\mbE\int_{0}^{T\wedge\tau_R^\veps}
\|\widetilde{u}(\sigma)-\bar{u}(\sigma)\|_{L^4(\T^d)}^4
\d\sigma\right)^{\frac14}\left(
\mbE\int_{0}^{T\wedge\tau_R^\veps}
\left(\|\widetilde{u}(\sigma)\|_{L^4(\T^d)}^4
+\|\bar{u}(\sigma)\|_{L^4(\T^d)}^4\right)\d\sigma
\right)^{\frac34}\\\nonumber
&\quad
+C\left(\mbE\int_{0}^{T\wedge\tau_R^\veps}
\|\widetilde{u}(\sigma)-\bar{u}(\sigma)\|_{L^4(\T^d)}^4
\d\sigma\right)^{\frac14}\left(\mbE\int_{0}^{T\wedge\tau_R^\veps}
\|\widetilde{u}^\veps(\sigma)\|_{L^4(\T^d)}^4\d\sigma
\right)^{\frac34}\\\nonumber
&
\leq R^{\frac18}C_{T,\kappa}\delta^{\frac{1}{32}}
\end{align}
by \eqref{uniest1}, \eqref{Hasteseq}, \eqref{esteqsp1} and \eqref{esteqsp2}.
Set $t(\delta):=\left[\frac{t}{\delta}\right]$ for all $0\leq t\leq T$.
For $\mathscr I_1^{2,2}$, in view of \eqref{esteqsp1}, \eqref{esteqsp2} and (G$_\gamma$), we have
\begin{align}\label{cubic08}
\mathscr I_1^{2,2}
&
=4\mbE\left(\sup_{0\leq t\leq T\wedge\tau_R^\veps}\sum_{k=0}^{t(\delta)-1}
\left(-\la|\bar{u}(k\delta)|^2\bar{u}(k\delta),
u^{\veps}(k\delta)-\bar{u}(k\delta)
\ra\int_{k\delta}^{(k+1)\delta}
\left(\gamma_\veps(\sigma)-\bar{\gamma}\right)\d\sigma
\right)\right)\\\nonumber
&\quad
+4\mbE\left(\sup_{0\leq t\leq T\wedge\tau_R^\veps}\int_{t(\delta)\delta}^{t}
-(\gamma_\veps(\sigma)-\bar{\gamma})\d\sigma\la
|\bar{u}(t(\delta)\delta)|^2\bar{u}(t(\delta)\delta),
u^{\veps}(t(\delta)\delta)
-\bar{u}(t(\delta)\delta)
\ra\right)\\\nonumber
&
\leq C\mbE\left(\sup_{0\leq t\leq T\wedge\tau_R^\veps}\sum_{k=0}^{t(\delta)-1}
\left(\|\bar{u}(k\delta)\|_{L^4(\T^d)}^4
+\|u^\veps(k\delta)\|_{L^4(\T^d)}^4\right)
\left|\int_{k\delta}^{(k+1)\delta}(\gamma_\veps(\sigma)-\bar{\gamma})
\d\sigma\right|\right)\\\nonumber
&\quad
+C_{T}\mbE\left(\sup_{0\leq t\leq T\wedge\tau_R^\veps}
\int_{t(\delta)\delta}^{t}
\left(\|\bar{u}(t(\delta)\delta)\|_{L^4(\T^d)}
\|u^\veps(t(\delta)\delta)\|^3_{L^4(\T^d)}
+\|\bar{u}(t(\delta)\delta)\|^4_{L^4(\T^d)}\right)\d\sigma\right)
\\\nonumber
&
\leq C_{T,\kappa}\delta_\gamma(\delta/\veps)
+C_{T}(\Xi_1+\Xi_2),
\end{align}
where
$$\Xi_1:=\mbE\left(\sup\limits_{0\leq t\leq T\wedge\tau_R^\veps}\int_{t(\delta)\delta}^t
\|\bar{u}(t(\delta)\delta)\|_{L^4(\T^d)}^4\d\sigma\right)$$
and
$$\Xi_2:=\mbE\left(\sup\limits_{0\leq t\leq T\wedge\tau_R^\veps}\int_{t(\delta)\delta}^t
\|u^\veps(t(\delta)\delta)\|^4_{L^4(\T^d)}
\d\sigma\right).
$$
Now we just show that $\Xi_2\leq R^{1/2}C_{T,\kappa}\delta^{1/8}$;
the case of $\Xi_1$ is similar.
By \eqref{Horsteseq}, \eqref{uniest1}, \eqref{uniest2} and H\"older's inequality, one sees that
\begin{align*}%\label{cubic09}
\Xi_2
&
\leq C\mbE\left(\sup_{0\leq t\leq T\wedge\tau_R^\veps}
\int_{t(\delta)\delta}^t\left(
\|u^\veps(\sigma)-u^\veps(t(\delta)\delta)\|^4_{L^4(\T^d)}
+\|u^\veps(\sigma)\|^4_{L^4(\T^d)}
\right)\d\sigma\right)\\\nonumber
&
\leq
R^{\frac12}C_{T,\kappa}\delta^{\frac{1}{8}}+\mbE\left(\sup_{0\leq t\leq T}
\int_{t(\delta)\delta}^t\|u^\veps(\sigma)\|
\|u^\veps(\sigma)\|^3_{1}\d\sigma\right)\\\nonumber
&
\leq R^{\frac12}C_{T,\kappa}\delta^{\frac{1}{8}}
+\mbE\left[\sup_{0\leq t\leq T}\left(\int_{t(\delta)\delta}^t
\|u^\veps(\sigma)\|^4\d\sigma\right)^{\frac14}\left(
\int_{0}^T\|u^\veps(\sigma)\|^4_{1}\d\sigma
\right)^{\frac34}\right]\\\nonumber
&
\leq R^{\frac12}C_{T,\kappa}\delta^{\frac{1}{8}}
+\delta^{\frac14}\mbE\left[\left(\sup_{0\leq t\leq T}\|u^\veps(\sigma)\|
\right)\left(
\int_{0}^T\|u^\veps(\sigma)\|^4_{1}\d\sigma
\right)^{\frac34}\right]\\\nonumber
&
\leq R^{\frac12}C_{T,\kappa}\delta^{\frac{1}{8}}
+\delta^{\frac14}\left(\mbE\sup_{0\leq t\leq T}\|u^\veps(\sigma)\|^4
\right)^{\frac14}\left(\mbE
\int_{0}^T\|u^\veps(\sigma)\|^4_{1}\d\sigma
\right)^{\frac34}\\\nonumber
&
\leq R^{\frac12}C_{T,\kappa}\delta^{\frac{1}{8}}.
\end{align*}
Hence, \eqref{cubic06}, \eqref{cubic07} and \eqref{cubic08} yield
\begin{align}\label{cubic10}
\mathscr I_1^2\leq C_{T,\kappa}\left(R^{\frac18}\delta^{\frac{1}{32}}
+\delta_\gamma(\delta/\veps)\right).
\end{align}
Combining \eqref{cubic03}, \eqref{cubic04}, \eqref{cubic05}
and \eqref{cubic10}, we have
\begin{align}\label{cubic11}
\mathscr I_1\leq C_{T,\kappa}\left(R^{\frac18}\delta^{\frac{1}{32}}
+\delta_\gamma(\delta/\veps)\right).
\end{align}

Now we estimate
$$\mathscr I_{2}:=4\mathbb E\left(\sup\limits_{0\leq t\leq T}\int_{0}^{t}\langle
f_{\veps}(\sigma,u^{\veps}(\sigma))
-\bar{f}(\bar{u}(\sigma)),u^{\veps}(\sigma)-\bar{u}(\sigma)
\rangle\d\sigma\right)$$
and
$$\mathscr I_{3}:=C\mbE\int_{0}^{T}\|g_{\veps}(\sigma,u^{\veps}(\sigma))
-\bar{g}(\bar{u}(\sigma))\|_{L_{2}(U,L^{2}(\T^{d}))}^{2}
\d\sigma.$$
Since $f_\veps$ and $g_\veps$ are Lipschitz continuous, by a similar argument as in \cite[Theorem 4.5]{CLAM2021}, we have
\begin{align}\label{avethfeq10}
\mathscr I_{2}
\leq4L_{f}\int_{0}^{T}\mbE\left(\sup_{0\leq\tau\leq\sigma}
\|u^{\veps}(\tau)-\bar{u}(\tau)\|^{2}\right)\d\sigma
+C_{T}\left(1+\mbE\|\zeta\|^{2}\right)
\left(\delta^{\frac{1}{4}}+
\delta_{f}\left(\frac{\delta}{\veps}\right)\right)
\end{align}
and
\begin{align}\label{avethfeq13}
\mathscr I_{3}\leq C\int_{0}^{T}\mbE\left(\sup_{s\leq\tau\leq\sigma}
\|u^{\veps}(\tau)-\bar{u}(\tau)\|^{2}\right)\d\sigma
+C_{T}\left(1+\mbE\|\zeta\|^{2}\right)\left(\delta^{\frac12}
+\delta_{g}\left(\frac{\delta}{\veps}\right)\right).
\end{align}

Combining \eqref{stop0}, \eqref{avethfeq3}, \eqref{cubic11}, \eqref{avethfeq10} and
\eqref{avethfeq13}, we get
\begin{align*}
\mbE\left(\sup_{0\leq t\leq T}\|u^{\veps}(t)-\bar{u}(t)\|^{2}\right)
&
\leq2\mbE\|\zeta^{\veps}-\zeta\|^{2}
+C\int_{0}^{T}\mbE\left(\sup_{s\leq\tau\leq \sigma}
\|u^{\veps}(\tau)-\bar{u}(\tau)\|^{2}\right)\d\sigma\\\nonumber
&\quad
+C_{T,\kappa}\left(
R^{\frac18}\delta^{\frac{1}{32}}
+\delta_{\gamma}\left(\frac{\delta}{\veps}\right)
+\delta_{f}\left(\frac{\delta}{\veps}\right)
+\delta_{g}\left(\frac{\delta}{\veps}\right)\right)
+R^{-2}C_{T,\kappa}.
\end{align*}
It follows from Gronwall's lemma that
\begin{align}\label{avethfeq15}
&
\mbE\left(\sup_{0\leq t\leq T}
\|u^{\veps}(t)-\bar{u}(t)\|^{2}\right)\\\nonumber
&
\leq C_{T,\kappa}\bigg(\mbE\|\zeta^{\veps}-\zeta\|^{2}
+\left(
R^{\frac18}\delta^{\frac{1}{32}}
+\delta_{\gamma}\left(\frac{\delta}{\veps}\right)
+\delta_{f}\left(\frac{\delta}{\veps}\right)
+\delta_{g}\left(\frac{\delta}{\veps}\right)+R^{-1}\right)
\bigg).
\end{align}
Taking $R=\veps^{-\frac{1}{136}}$ and $\delta=\sqrt{\veps}$, then letting
$\veps\rightarrow0$ in \eqref{avethfeq15}, we have
\[
\lim_{\veps\rightarrow0}\mbE\left(\sup_{0\leq t\leq T}
\|u^{\veps}(t)-\bar{u}(t)\|^{2}\right)=0.
\]
\end{proof}

\begin{remark}\label{gamcon}\rm
Note that if $\gamma(t)\equiv\gamma$ is independent of time,
we need not to estimate $\mathscr I_{1}$. Therefore,
Theorem \ref{avethf} still holds without assumptions \eqref{conkappa}
and (G$_g^2$).
\end{remark}

\section{The second Bogolyubov theorem}\label{secondBT}

In this section, we study the second Bogolyubov theorem for
stochastic CGL equations. Firstly, we show that there exists
a unique $\mcL^{2}(\Omega,\mathbb P;L^{2}(\T^{d}))$-bounded
solution $u^{\veps}(t),t\in\R$ of \eqref{eqG2.1} which
inherits the recurrent properties (in particular, periodic,
quasi-periodic, almost periodic, almost automorphic, Birkhoff
recurrent, Levitan almost periodic, almost recurrent,
pseudo-periodic, pseudo-recurrent, Poisson stable)
of the coefficients in distribution sense for any $0<\veps\leq1$.
This result is interesting in its own right and has been studied extensively;
see e.g. \cite{CL_2017, CML2020, DT1995, LL2020} and references therein.
Since our main results are the averaging principles, we omit the definitions
of these recurrent functions for brevity; see \cite{CL_2017, CLAM2021, LL2020}
for details. Then we prove that the recurrent solution
to the original equation converges to the stationary solution of the averaged
equation when the time scale goes to zero.
An $H$-valued stochastic process $X(t),t\in\R$ is called
{\em $\mcL^{p}(\Omega,\P;H)$-bounded} if
\[
\sup\limits_{t\in\R}\mbE\|X(t)\|_H^{p}<\infty.
\]

\subsection{Bounded solution}\label{boundedsol}
Without loss of generality, we assume that $\veps=1$ in this subsection;
that is to say, we consider the equation \eqref{maineq1}, i.e.
\begin{equation*}
 \d u(t)=\left[(1+i\alpha )\Delta u(t)-(\gamma(t)+i\beta)|u(t)|^{2}u(t)
     +f(t,u(t))\right]\d t + g(t,u(t))\d W(t).
  \end{equation*}
Now, we introduce the following condition
\begin{enumerate}
  \item[{\bf(H$_f^3$)}] There exists a constant $\lambda_f\in\R$ such that for all $t\in\R$ and $x,y\in L^2(\T^d)$
      \[
      \la f(t,x)-f(t,y),x-y\ra\leq\lambda_f\|x-y\|^2.
      \]
\end{enumerate}

\begin{remark}\rm
Note that (H$_f^1$) implies (H$_f^3$). In order to study the bounded solution,
the system \eqref{maineq1} needs to be dissipative, i.e.
$\lambda_*-\lambda_f-\frac{L_g^2}{2}>0$. Since condition $\lambda_*-L_f-\frac{L_g^2}{2}>0$
is stronger than $\lambda_*-\lambda_f-\frac{L_g^2}{2}>0$
when $\lambda_f$ is negative, we introduce (H$_f^3$) additionally.
\end{remark}

\begin{lemma}\label{essol}
Let $\gamma(\cdot)\geq\frac{|\beta|}{\sqrt3}$.
Assume that {\rm{(H$_f^1$)}}, {\rm{(H$_f^3$)}} and {\rm{(H$_g^1$)}} hold,
$\lambda_{*}-\lambda_{f}-\frac{L_{g}^{2}}{2}>0$.
Fix $s\in\R$. Let $u(t,s,\zeta_{s})$, $t\geq s$ be the solution to
\eqref{maineq1} with the initial data $\zeta_s$.
Then we have the following statements.
\begin{enumerate}
  \item If $\zeta_{s}\in \mcL^{2p}(\Omega,\mathbb P;L^{2}(\T^{d}))$, then
  for any $\eta\in(0,2\lambda_{*}-2\lambda_{f}-(2p-1)L_{g}^{2})$
there exists a constant $M_{1}>0$ such that
\begin{equation}\label{ineq2}
\mathbb E\|u(t,s,\zeta_{s})\|^{2p}
\leq {\rm{e}}^{-\eta p(t-s)}\mathbb E\|\zeta_{s}\|^{2p}+M_{1},
\end{equation}
where $1\leq p<\frac{\lambda_*-\lambda_f}{L_g^2}+\frac12$, $M_1$ depends only on $\eta$, $p$ and $K$.
  \item If {\rm{(H$_g^2$)}} holds and
$\zeta_{s}\in \mcL^{2p}(\Omega,\mathbb P;H_0^{1})$, then
there exists a constant $M_{2}>0$ such that
\begin{equation}\label{ineq3}
\mathbb E\|u(t,s,\zeta_{s})\|_{1}^{2p}
\leq M_{2}\left(
{\rm{e}}^{-\eta p(t-s)}\mathbb E\|\zeta_{s}\|_{1}^{2p}+1\right),
\end{equation}
where $p$ and $\eta$ are as in {\rm(i)}, $M_2$ depends only on $\eta$, $p$ and $K$.
  \item
  If {\rm{(H$_f^2$)}}, {\rm{(H$_g^2$)}} and {\rm{(H$_g^3$)}} hold,
  $\lambda_*-\lambda_f-\frac{9}{2}L_g^2>0$ and
$\zeta_{s}\in \mcL^{2}(\Omega,\mathbb P;H_0^{2})\cap\mcL^{10}(\Omega,\mathbb P;H_0^{1})$,
then there exists a constant $M_{3}>0$ such that
\begin{equation}\label{ineq4}
\mathbb E\|u(t,s,\zeta_{s})\|_{2}^{2}
\leq M_{3}\left(
{\rm{e}}^{-\eta'(t-s)}\mathbb E\|\zeta_{s}\|_{2}^{2}+1\right),
\end{equation}
where $\eta'\in(0,2\lambda_*-2\lambda_f-9L_g^2)$, $M_{3}$ depends only on $\eta$, $p$, $K$ and $\eta'$.
\end{enumerate}
\begin{proof}
(i)
By the product rule, It\^o's formula, (H$_f^3$), (H$_g^1$) and Young's inequality, we have
\begin{align*}
&
\mathbb E\left({\rm{e}}^{\eta p(t-s)}
\|u(t,s,\zeta_{s})\|^{2p}\right)\\\nonumber
&
=\mathbb E\|\zeta_{s}\|^{2p}
+\int_{s}^{t}\eta p{\rm{e}}^{\eta p(\sigma-s)}
\mathbb E\|u(\sigma,s,\zeta_{s})\|^{2p}\d\sigma\\\nonumber
&\quad
+p\mathbb E\int_{s}^{t}\|u(\sigma,s,\zeta_{s})\|^{2p-2}
{\rm{e}}^{\eta p(\sigma-s)}\Big(2\langle(1+i\alpha )
\Delta u(\sigma,s,\zeta_{s}),
u(\sigma,s,\zeta_{s})\rangle\\\nonumber
&\qquad
-2\langle(\gamma(\sigma)+i\beta)|u(\sigma,s,\zeta_{s})|^{2}u(\sigma,s,\zeta_{s}),
u(\sigma,s,\zeta_{s})\rangle\\\nonumber
&\qquad
+2\langle f(\sigma,u(\sigma,s,\zeta_{s})),
u(\sigma,s,\zeta_{s})\rangle
+\|g(\sigma,u(\sigma,s,\zeta_{s}))\|^{2}_{L_{2}(U,L^{2}(\T^{d}))}
\Big)\d\sigma
\\\nonumber
&\quad
+2p(p-1)\mathbb E\int_{s}^{t}{\rm{e}}^{\eta p(\sigma-s)}
\|u(\sigma,s,\zeta_{s})\|^{2p-4}\|\left(
g(\sigma,u(\sigma,s,\zeta_{s}))\right)^{*}
u(\sigma,s,\zeta_{s})\|_{U}^{2}\d\sigma\\\nonumber
&
\leq\mathbb E\|\zeta_{s}\|^{2p}
+\int_{s}^{t}\eta p{\rm{e}}^{\eta p(\sigma-s)}
\mathbb E\|u(\sigma,s,\zeta_{s})\|^{2p}\d\sigma\\\nonumber
&\quad
+p\mathbb E\int_{s}^{t}\|u(\sigma,s,\zeta_{s})\|^{2p-2}
{\rm{e}}^{\eta p(\sigma-s)}
\left(-\theta\|u(\sigma,s,\zeta_{s})\|^{2}
+C_{\veps}\right)\d\sigma,\\\nonumber
\end{align*}
where $\theta:=2\lambda_{*}-2\lambda_{f}-(2p-1)L_{g}^{2}-\veps-\veps L_{g}^{2}$.
When $\veps$ is small enough, it follows from Young's inequality
that there exists constant $M_{1}>0$ such that
$$
\mathbb E\|u(t,s,\zeta_{s})\|^{2p}
\leq {\rm{e}}^{-\eta p(t-s)}\mathbb E\|\zeta_{s}\|^{2p}+M_{1}
$$
for all $1\leq p<\frac{\lambda_*-\lambda_f}{L_g^2}+\frac12$,
where $M_1$ depends on $\eta$, $p$ and $K$.

(ii)
Recall that $u^n(t),t\geq s$ is the solution to
the finite dimensional equation \eqref{GFeq} for all $n>0$.
In view of \eqref{esteqH1}, (H$_f^1$), (H$_g^2$) and Young's inequality, for small $\veps$ we have
\begin{align*}
&
\mathbb E\|u^{n}(t)\|_{1}^{2p}\\
&
\leq\mathbb E\|\zeta_{s}\|_{1}^{2p}+\mathbb E\int_{s}^{t}
p\|u^{n}(\sigma)\|_{1}^{2p-2}
\Big(-2\|\Delta u^{n}(\sigma)\|^{2}
+\veps\|\Delta u^{n}(\sigma)\|^{2}+C_{\veps}
\|f(\sigma,u^{n}(\sigma))\|^{2}\\
&\qquad
+(2p-1)L_{g}^{2}\|u^{n}(\sigma)\|_{1}^{2}
+\veps\|g(\sigma,u^{n}(\sigma))\|_{L_{2}(U,H_0^{1})}^{2}
+C_{\veps}\Big)\d\sigma\\
&
\leq\mathbb E\|\zeta_{s}\|_{1}^{2p}+\mathbb E\int_{s}^{t}
p\|u^{n}(\sigma)\|_{1}^{2p-2}
\Big(\left(-2+\veps\right)\|\Delta u^{n}(\sigma)\|^{2}
+C_{\veps}\|u^{n}(\sigma)\|^{2}\\\nonumber
&\qquad
+\left((2p-1)L_{g}^{2}+\veps L_{g}^{2}\right)
\|u^{n}(\sigma)\|_{1}^{2}+C_{\veps}
\Big)\d\sigma\\
&
\leq\mathbb E\|\zeta_{s}\|_{1}^{2p}+\mathbb E\int_{s}^{t}\left(-\widetilde{\eta}p
\|u^{n}(\sigma)\|_{1}^{2p}
+C_{\veps}\|u^{n}(\sigma)\|^{2p}
+C_{\veps}\right)\d\sigma,
\end{align*}
where $\widetilde{\eta}:=2\lambda_{*}-\veps\lambda_{*}
-(2p-1)L_{g}^{2}-\veps L_{g}^{2}-\veps$. Then one sees that
\begin{align*}%\label{ineq3.1}
\mathbb E\|u^{n}(t)\|_{1}^{2p}
\leq{\rm{e}}^{-\widetilde{\eta}p(t-s)}
\mathbb E\|\zeta_{s}\|_{1}^{2p}
+\int_{s}^{t}C_{\veps}
{\rm{e}}^{-\widetilde{\eta}p(t-\sigma)}
\mathbb E\|u^{n}(\sigma)\|^{2p}\d\sigma
+\frac{C_{\veps}}{\widetilde{\eta}p}
-\frac{C_{\veps}}{\widetilde{\eta}p}{\rm{e}}^{-\widetilde{\eta}p(t-s)}.
\end{align*}
Therefore, taking $\veps$ small enough such that
$\widetilde{\eta}>\eta$, by \eqref{ineq2} we get
\begin{align}\label{fh1eq}
\mathbb E\|u^n(t)\|_{1}^{2p}
&
\leq{\rm{e}}^{-\widetilde{\eta}p(t-s)}
\mathbb E\|\zeta_{s}\|_{1}^{2p}
+\int_{s}^{t}C{\rm{e}}^{-\widetilde{\eta}p(t-\sigma)}\left(
\mathbb E\|\zeta_{s}\|^{2p}{\rm{e}}^{-\eta p(\sigma-s)}
+1\right)\d\sigma+\frac{C}{\widetilde{\eta}p}\\\nonumber
&
\leq{\rm{e}}^{-\widetilde{\eta}p(t-s)}
\mathbb E\|\zeta_{s}\|_{1}^{2p}
+C\mathbb E\|\zeta_{s}\|^{2p}{\rm{e}}^{-\widetilde{\eta}pt+\eta ps}
\int_{s}^{t}{\rm{e}}^{(\widetilde{\eta}-\eta)p\sigma}\d\sigma
+\frac{C}{\widetilde{\eta}p}\\\nonumber
&
\leq{\rm{e}}^{-\widetilde{\eta}p(t-s)}
\mathbb E\|\zeta_{s}\|_{1}^{2p}
+C\mathbb E\|\zeta_{s}\|^{2p}{\rm{e}}^{-\widetilde{\eta}pt+\eta ps}
\frac{1}{(\widetilde{\eta}-\eta)p}{\rm{e}}^{(\widetilde{\eta}-\eta)pt}
+\frac{C}{\widetilde{\eta}p}\\\nonumber
&
\leq M_{2}\left(\mathbb E\|\zeta_{s}\|_{1}^{2p}
{\rm{e}}^{-\eta p(t-s)}+1\right),
\end{align}
where $M_2$ depends on $p,K,\eta$.
Employing the reflexivity of $\mcL^{2p}(\Omega,\mbP;H^1)$ and \eqref{fh1eq}, we have
\[
\mathbb E\|u(t,s,\zeta_{s})\|_{1}^{2p}
\leq\liminf_{n\rightarrow\infty}\mathbb E\|u^{n}(t)\|_{1}^{2p}
\leq M_2\left(\mbE\|\zeta_s\|_1^{2p}{\rm e}^{-\eta p(t-s)}+1\right).
\]

(iii) By \eqref{H202}, we have
\begin{align*}%\label{H204}
\mbE\|u^n(t)\|_{2}^2
&
\leq\mbE\|\zeta_s\|_{2}^2+\mbE\int_s^t\left(-\widetilde{\widetilde{\eta}}\|u^n(\sigma)\|_{2}^2
+C_\veps\|u^n(\sigma)\|_{1}^{10}
+C_\veps\right)\d\sigma,
\end{align*}
where $\widetilde{\widetilde{\eta}}:=(2-\veps)\lambda_*
-(1+\veps)L_g^2-\veps>\eta'$ for $\veps$ small enough.
Since $\frac{\lambda_*-\lambda_f}{L_g^2}+\frac12>5$,
in view of \eqref{fh1eq}, one sees that
\begin{align*}
\mbE\|u^n(t)\|_{2}^2
&
\leq{\rm{e}}^{-\widetilde{\widetilde{\eta}}(t-s)}\mbE\|\zeta_s\|_{2}^2
+\int_s^t C{\rm{e}}^{-\widetilde{\widetilde{\eta}}(t-\sigma)}
\mbE\|u^n(\sigma)\|_{1}^{10}\d\sigma+C\\\nonumber
&
\leq M_3\left(\left(\mbE\|\zeta_s\|_{2}^2+\mbE\|\zeta_s\|_{1}^{10}\right)
{\rm e}^{-\eta'(t-s)}+1\right),
\end{align*}
which completes the proof by the reflexivity of $\mcL^2(\Omega,\mbP;H_0^2)$.
\end{proof}
\end{lemma}

\begin{prop}\label{Boundedth}
Consider equation \eqref{maineq1}. Assume that conditions
{\rm{(H$_f^1$)}}, {\rm{(H$_f^3$)}} and {\rm{(H$_g^1$)}} hold. Suppose further that
$\lambda_{*}-\lambda_{f}-\frac{L_{g}^{2}}{2}>0$ and
$\gamma(t)\geq\frac{|\beta|}{\sqrt3}$ for all $t\in\R$. Then there is a
unique $\mcL^{2}(\Omega,\P;L^{2}(\T^{d}))$-bounded solution
$u(t)$, $t\in\mathbb{R}$ to equation \eqref{maineq1},
and the mapping
$\widehat{\mu}:\mathbb{R}\rightarrow Pr_{2}(L^{2}(\T^{d}))$,
defined by $\widehat{\mu}(t):=\mathbb P\circ[u(t)]^{-1}$,
is unique with flow property, i.e.
$\mu(t,s,\widehat{\mu}(s))=\widehat{\mu}(t)$
for all $t\geq s$. Moreover,
\begin{enumerate}
  \item if {\rm{(H$_g^2$)}} holds, then
  \begin{equation}\label{BH01}
  \sup_{t\in\R}\mbE\|u(t)\|_{1}^{2p}<\infty,
  \end{equation}
  where $1\leq p<\frac{\lambda_*-\lambda_f}{L_g^2}+\frac12$;
  \item
  if {\rm{(H$_f^2$)}}, {\rm{(H$_g^2$)}}, {\rm{(H$_g^3$)}} hold and
$\lambda_*-\lambda_f-\frac{9}{2}L_g^2>0$, then
\begin{equation}\label{BH02}
\sup_{t\in\R}\mbE\|u(t)\|_{2}^{2}<\infty,
\end{equation}
\end{enumerate}
Here $\mu(t,s,\mu_{0})$ denotes the distribution of $u(t,s,\zeta_{s})$
on $L^{2}(\T^{d})$, with $\mu_{0}=\mathbb P\circ\zeta^{-1}_{s}$.
\end{prop}
\begin{proof}
Let $u_1(\cdot):=u_1(\cdot,s,\zeta^1_{s})$ and $u_2(\cdot):=u_2(\cdot,s,\zeta^2_{s})$
be two solutions to \eqref{maineq1}.
Define $\theta:=2\lambda_{*}-2\lambda_{f}-L_{g}^{2}$.
In view of It\^o's formula, the product rule, (H$_f^3$) and (H$_g^1$), we obtain
\begin{align*}
&
\mathbb E\left({\rm{e}}^{\theta(t-s)}
\|u_1(t)-u_2(t)\|^{2}\right)\\\nonumber
&
=\mathbb E\|\zeta^1_{s}-\zeta_s^2\|^{2}
+\int_{s}^{t}\theta
{\rm{e}}^{\theta(\sigma-s)}
\mathbb E\|u_1(\sigma)-u_2(\sigma)\|^{2}\d\sigma \\\nonumber
&\quad
+\mathbb E\int^{t}_{s}{\rm{e}}^{\theta(\sigma-s)}\Big(2\langle (1+i\alpha )\Delta\left(u_1(\sigma)
-u_2(\sigma)\right),u_1(\sigma)
-u_2(\sigma)\rangle\\\nonumber
&\qquad
-2\langle(\gamma(\sigma)+i\beta)\left(|u_1(\sigma)|^{2}
u_1(\sigma)
-|u_2(\sigma)|^{2}u_2(\sigma)\right),
u_1(\sigma)-u_2(\sigma)\rangle\\\nonumber
&\qquad
+2\langle f(\sigma,u_1(\sigma))
-f(\sigma,u_2(\sigma)),
u_1(\sigma)-u_2(\sigma)\rangle\\\nonumber
& \qquad
+\|g(\sigma,u_1(\sigma))
-g(\sigma,u_2(\sigma))\|^{2}_{L_{2}(U,L^{2}(\T^{d}))}
\Big)\d\sigma
\leq \mathbb E\|\zeta^1_{s}-\zeta^2_{s}\|^{2}.
\end{align*}
Hence for all $t\geq s$
\begin{equation}\label{gasmseq00}
  \mathbb E\|u_1(t)-u_2(t)\|^{2}\leq
  {\rm{e}}^{-(2\lambda_{*}-2\lambda_{f}-L_{g}^{2})(t-s)}
  \mathbb E\|\zeta^1_{s}-\zeta_s^2\|^{2},
\end{equation}
which implies that
\begin{align*}
\mathbb E\|u_{n}(t)-u_{m}(t)\|^{2}
\leq M_{1}{\rm{e}}^{-(2\lambda_{*}-2\lambda_{f}-L_{g}^{2})(t+m)}
\end{align*}
for all $t\geq-m\geq-n$ by \eqref{ineq2},
where $u_n(t):=u(t,-n,0)$ for all $n\in\N^+$.
Letting $n>m$, $m\rightarrow\infty$, we have
$$\mathbb E\|u_{n}(t)-u_{m}(t)\|^{2}\rightarrow 0.$$
Since $\mcL^{2}(\Omega,\mathbb P;L^{2}(\T^{d}))$ is complete,
there exists a process $u(t)$, $t\in\R$ such that
\begin{equation}\label{Btheq1}
u_{n}(t)\rightarrow u(t) \quad {\rm {in}}~
\mcL^{2}(\Omega,\mathbb P;L^{2}(\T^{d}))
\end{equation}
for any  $t\in\R$.
And it follows from \eqref{ineq2} that
\begin{align*}
\sup\limits_{t\in\mathbb{R}}\mbE\|u(t)\|^{2}
\leq \sup_{t\in\R}\left(\liminf_{n\rightarrow\infty}\mbE\|u(t,-n,0)\|^2\right)
\leq M_{1},
\end{align*}
which implies that
\[
\sup_{t\in\mathbb{R}}\int_{L^{2}(\T^{d})}\|x\|^{2}\widehat{\mu}(t)
(\d x)<\infty,
\]
where $\widehat{\mu}(t):=\mathscr L(X(t)),t\in\R$.

Similar to the proof of Theorem \ref{wellD} in Appendix \ref{eupf},
we can also prove that the limit process $u(\cdot)$ in
\eqref{Btheq1} is a solution to equation \eqref{maineq1}.
The uniqueness of $\mcL^{2}(\Omega,\P;L^{2}(\T^{d}))$-bounded solution
follows from \eqref{gasmseq00}.

The next goal is to prove that
$\widehat{\mu}:\R\rightarrow Pr_{2}(L^{2}(\T^{d}))$ is unique with
flow property. In view of the Chapman-Kolmogorov equation and
the Feller property, we get
$$\mu(t,s,\widehat{\mu}(s))=\widehat{\mu}(t).$$
Suppose that $\mu_{1},\mu_{2}:\R\rightarrow Pr_{2}(L^{2}(\T^{d}))$
satisfy the flow property,
let $\zeta_{n,1}$ and $\zeta_{n,2}$ be random variables with
distributions $\mu_{1}(-n)$ and $\mu_{2}(-n)$ respectively.
Then consider  solutions $u(t,-n,\zeta_{n,1})$ and
$u(t,-n,\zeta_{n,2})$ on $[-n,\infty)$, we have
\begin{align*}
 W_{2}(\mu_{1}(t),\mu_{2}(t))
&
 =W_{2}(\mu(t,-n,\mu_{1}(-n)),\mu(t,-n,\mu_{2}(-n))) \\
&
 \leq \left(\mathbb E\|u(t,-n,\zeta_{n,1})
  -u(t,-n,\zeta_{n,2})\|^{2}\right)^{1/2} \\
&
 \leq {\rm{e}}^{-(\lambda_{*}-\lambda_{f}-\frac{L_{g}^{2}}{2})(t+n)}
 \left(\mathbb E\|\zeta_{n,1}-\zeta_{n,2}\|^{2}\right)^{1/2}
 \rightarrow0 \quad {\rm{as}}~n\rightarrow\infty.
\end{align*}
Thus, $\mu_{1}(t)=\mu_{2}(t)$ for all $t\in\R$.

Assume that (H$_f^2$), (H$_g^2$) and (H$_g^3$) hold. By \eqref{ineq3} and \eqref{ineq4}, we have
$$
\sup_{t\in\R}\mbE\|u_{n}(t)\|_{1}^{2p}+
\sup_{t\in\R}\mbE\|u_{n}(t)\|_{2}^{2}<\infty,
$$
where $1\leq p<\frac{\lambda_*-\lambda_f}{L_g^2}+\frac12$.
Then there exists a subsequence of $\{u_{n}(t)\}$ which we still denote by
$\{u_{n}(t)\}$ such that $u_n(t)\rightarrow u(t)$
weakly in $\mcL^{2p}(\Omega,\mathbb P;H_0^{1})$ and
$\mcL^{2}(\Omega,\mathbb P;H_0^{2})$ for all $t\in\R$.
Therefore, we have
$$
\sup_{t\in\R}\mbE\|u(t)\|_{1}^{2p}+
\sup_{t\in\R}\mbE\|u(t)\|_{2}^{2}<\infty.
$$
\end{proof}

\begin{remark}\rm
(i) It follows from \eqref{ineq2} that
\begin{equation}\label{B2p}
\sup_{t\in\R}\mbE\|u(t)\|^{2p}\leq\sup_{t\in\R}
\left(\liminf_{n\rightarrow\infty}\mbE\|u(t,-n,0)\|^{2p}\right)
<\infty,
\end{equation}
where $p>1$ is as in \eqref{ineq2}.

(ii) Note that this
$\mcL^{2}(\Omega,\mathbb P;L^{2}(\T^{d}))$-bounded solution is
$T$-periodic provided $f$ and $g$ are $T$-periodic.
The proof is similar to \cite[Theorem 4.1]{CML2020}.
\end{remark}

\begin{definition}[See \cite{FL}]\rm
Let $s\in\R$. A solution $u(t),t\geq s$ of equation \eqref{maineq1} is called
{\em stable in square-mean sense}, if for each $\epsilon>0$,
there exists $\delta>0$ such that for all $t\geq s$
\begin{equation*}
  \mathbb E\|u(t,s,\zeta_{s})-u(t)\|^{2}<\epsilon,
\end{equation*}
whenever $\mathbb E\|\zeta_{s}-u(s)\|^{2}<\delta$.
The solution $u(t),t\geq s$ is said to be {\em asymptotically
stable in square-mean sense} provided it is stable in square-mean sense and
\begin{equation}\label{stab}
  \lim_{t\rightarrow\infty}\mathbb E\|u(t,s,\zeta_{s})-u(t)\|^{2}=0.
\end{equation}
We say $u(t),t\geq s$ is {\em globally asymptotically stable in
square-mean sense} provided \eqref{stab} holds for any
$\zeta_{s}\in \mcL^{2}(\Omega,\mathbb P;L^{2}(\T^{d}))$.
\end{definition}

By \eqref{gasmseq00}, we have the following lemma, which states that
the $\mcL^{2}(\Omega,\P;L^{2}(\T^{d}))$-bounded solution of equation
\eqref{maineq1} is globally asymptotically stable in square-mean sense.
\begin{lemma}\label{gasms}
Consider equation \eqref{maineq1}.
Suppose that {\rm{(H$_f^1$)}}, {\rm{(H$_f^3$)}} and {\rm{(H$_g^1$)}} hold.
Assume further that $\gamma(t)\geq\frac{|\beta|}{\sqrt3}$ for all $t\in\R$ and
$\lambda_{*}-\lambda_{f}-\frac{L_{g}^{2}}{2}>0$. Then the unique
$\mcL^{2}(\Omega,\P;L^{2}(\T^{d}))$-bounded solution
$u(\cdot)$ of equation \eqref{maineq1} is
globally asymptotically stable in square-mean sense.
Moreover, let $s\in\R$. For any $t\geq s$ and $\zeta_{s}\in
\mcL^{2}(\Omega,\mathbb P;L^{2}(\T^{d}))$ we have
\begin{equation}\label{gasmseq}
  \mathbb E\|u(t,s,\zeta_{s})-u(t)\|^{2}
  \leq {\rm{e}}^{-(2\lambda_{*}-2\lambda_{f}-L_{g}^{2})(t-s)}
  \mathbb E\|\zeta_{s}-u(s)\|^{2}.
\end{equation}
\end{lemma}

\subsection{The second Bogolyubov theorem}
In this subsection, we firstly show that the
$\mcL^{2}(\Omega,\P;L^{2}(\T^{d}))$-bounded solution for equation \eqref{maineq1} is
strongly compatible in distribution (see Definition \ref{compatible}),
which implies that this bounded solution inherits the
recurrent properties of the coefficients in
distribution sense by Shcherbakov's comparability method; see Section 2.3 in \cite{CLAM2021} for details
about this method.
Then we prove that this bounded solution for \eqref{eqG2.1}
converges to the stationary solution for \eqref{eqG5_1}
when $\veps$ goes to zero. For this,
we need the following condition.
\begin{itemize}
\item [\textbf{(H$_c$)}] $f$ and $g$ are continuous in $t$ uniformly
  with respect to $x$ on each bounded subset
  $Q\subset L^{2}(\T^{d})$.
\end{itemize}

\begin{remark}\rm\label{hulllem}
(i) If $f$ and $g$ satisfy (H$_f^1$), (H$_f^2$), (H$_f^3$), (H$_g^1$), (H$_g^2$)
    and (H$_g^3$), then every pair of functions
    $\left(\widetilde{f},\widetilde{g}\right)\in\mathcal H(f,g)$
    possess the same property with the same constants. Recall that
      \[
      \mathcal H(f,g):=\overline{\left\{\left(f^{\tau},g^{\tau}\right):
      \tau\in\R\right\}},
      \]
    where $f^{\tau}$ is the {\em$\tau$-translation} of $f$ defined by
    $f^{\tau}(t,x):=f(t+\tau,x)$ for all $t\in\R$ and $x\in L^{2}(\T^{d})$.

(ii) If $f$ and $g$ satisfy the conditions (H$_f^1$), (H$_g^1$) and (H$_c$), then
      $f\in BUC(\R\times L^{2}(\T^{d}), L^{2}(\T^{d}))$,
      $g\in BUC(\R\times L^{2}(\T^{d}),L_{2}(U,L^{2}(\T^{d})))$
      and
      $\mathcal H(\gamma,f,g)\subset C(\R,\R_+)\times BUC(\R\times L^{2}(\T^{d}), L^{2}(\T^{d}))
      \times BUC(\R\times L^{2}(\T^{d}),L_{2}(U,L^{2}(\T^{d})))$
      is a shift dynamical system. See Appendix \ref{shiftDS} for more details about
      the space $BUC$ and shift dynamical systems.

(iii) By Remark \ref{reuniest},
\eqref{ineq2}--\eqref{ineq3} and \eqref{BH01} hold
for the averaged equation \eqref{eqG5_1}.
\end{remark}

Let $\varphi\in C(\R,\mathcal X)$. Denote by $\mathfrak{N}_{\varphi}$
(respectively, $\mathfrak{M}_{\varphi}$) the space of all sequences
$\{t_{n}\}_{n=1}^{\infty}$ such that $\varphi(\cdot+t_{n})$ converges
to $\varphi(\cdot)$ (respectively, $\varphi(\cdot+t_{n})$ converges)
uniformly on any compact interval.

\begin{definition}\label{compatible} \rm
Let $\varphi (t),t\in\mathbb R$ be a solution of
equation \eqref{maineq1}. Then $\varphi$ is called {\em
compatible} (respectively, {\em strongly compatible}) {\em in
distribution} if the following conditions are fulfilled:
\begin{enumerate}
\item
there exists a bounded closed subset
$\mathcal Q\subset L^{2}(\Omega,\mathbb P;L^{2}(\T^{d})$)
such that
$\varphi(\mathbb R)\subseteq \mathcal Q$;
\item
$\mathfrak N_{(f,g)}\subseteq \widetilde{\mathfrak N}_{\varphi}$
(respectively,
$\mathfrak M_{(f,g)}\subseteq \widetilde{\mathfrak M}_{\varphi}$),
where $\widetilde{\mathfrak N}_{\varphi}$
(respectively, $\widetilde{\mathfrak M}_{\varphi}$) means the set of
all sequences $\{t_n\}\subset\mathbb R$ such that the sequence
$\{\varphi(\cdot+t_n)\}$ converges to $\varphi(\cdot)$
(respectively, $\{\varphi(\cdot+t_n)\}$ converges) in distribution
uniformly on any compact interval.
\end{enumerate}
\end{definition}

With the help of the above estimates, we can show that
the bounded solution is compatible in distribution and establish
the second Bogolyubov theorem for stochastic CGL equations;
the idea and argument are similar to \cite{CLAM2021}.
For brevity, we state these results without proof.
\begin{theorem}\label{compth}
Let $\gamma(\cdot)\geq\frac{|\beta|}{\sqrt3}$.
Suppose that {\rm(H$_f^1$)}, {\rm(H$_f^3$)}, {\rm(H$_g^1$)},
{\rm(H$_g^2$)} and {\rm(H$_c$)} hold and
$\lambda_{*}-\lambda_{f}-\frac{L_{g}^{2}}{2}>0$. Then the unique
$\mcL^{2}(\Omega,\P;L^{2}(\T^{d}))$-bounded solution $u(\cdot)$
of \eqref{maineq1} is strongly compatible in distribution.
\end{theorem}

\begin{theorem}\label{averth}
Let $\gamma(\cdot)\geq\frac{|\beta|}{\sqrt3}$.
Suppose that $f$ and $g$ satisfy  {\rm(H$_f^1$)}--{\rm(H$_f^3$)},
{\rm(H$_g^1$)}--{\rm(H$_g^3$)}, {\rm(H$_c$)},
{\rm(G$_\gamma$)}, {\rm(G$_f$)} and {\rm{(G$_g^1$)--(G$_g^2$)}}. If
$\lambda_{*}-\lambda_{f}-\frac{9}{2}L_{g}^{2}>0$,
then for any $0<\veps \le 1$
\begin{enumerate}
\item equation \eqref{eqG2.1} has a unique solution
      $u^{\veps}\in C_{b}(\mathbb R,\mcL^{2}
      (\Omega,\mathbb P;L^{2}(\T^{d})))$;
\item the solution $u^{\veps}$ is
      strongly compatible in distribution and
      \[
      \lim_{\veps\rightarrow0}\sup\limits_{s\leq t\leq s+T}
      W_{2}(\mathscr L(u^{\veps}(t)),
      \mathscr L(\bar{u}(0)))=0
      \]
      for all $s\in\R$ and $T>0$,
      where $\bar{u}$ is the unique stationary solution of the
      averaged equation \eqref{eqG5_1}.
\end{enumerate}
\end{theorem}

\begin{remark}\label{corL3}\rm
Under the conditions of Theorem \ref{averth}, the following statements hold.
\begin{enumerate}
\item
If $\gamma\in C(\R,\R_+)$, $f\in$ $ C(\mathbb R\times L^{2}(\T^{d}),L^{2}(\T^{d}))$ and
$g\in C(\R\times L^{2}(\T^{d}),L_{2}(U,L^{2}(\T^{d})))$ are jointly
stationary (respectively, $T$-periodic, quasi-periodic with the
spectrum of frequencies $\nu_1,\ldots,\nu_k$, almost
periodic, almost automorphic, Birkhoff recurrent, Lagrange
stable, Levitan almost periodic, almost recurrent, Poisson
stable) in $t$ uniformly with respect to $x$ on each bounded
subset, then so is the unique solution
$u^{\veps} \in C_{b}(\mathbb R,\mcL^2(\Omega,\mathbb P;L^{2}(\T^{d})))$
of \eqref{eqG2.1} in distribution;
\item
If $\gamma\in C(\R,\R_+)$, $f\in$ $ C(\mathbb R\times L^{2}(\T^{d}),L^{2}(\T^{d}))$ and
$g\in C(\R\times L^{2}(\T^{d}),L_{2}(U,L^{2}(\T^{d})))$
are Lagrange stable and jointly pseudo-periodic
(respectively, pseudo-recurrent) in $t$ uniformly with respect to
$x$ on each bounded subset, then the unique
$\mcL^2(\Omega,\mathbb P;L^{2}(\T^{d}))$-bounded solution
$u^{\veps}$ of \eqref{eqG2.1} is pseudo-periodic
(respectively, pseudo-recurrent) in distribution.
\end{enumerate}
\end{remark}

\section{Global averaging principle in weak sense}\label{globalAP}

Let $u(t,s,x),t\geq s$ be the solution of equation \eqref{maineq1}
with initial value $u(s,s,x)=x$.
Recall that
\[
P^{*}(t,F,\mu)(A)=\int_{L^{2}(\T^{d})}
P(0,x,t,A)\mu(\d x)
=\int_{L^2(\T^d)}\mathbb P\circ\left(
u(t,0,x)\right)^{-1}(A)\mu(\d x)
\]
for all $\mu\in Pr(L^{2}(\T^{d}))$ and $A\in\mathcal{B}(L^{2}(\T^{d}))$,
where $F=(\gamma,f,g)$.

In this section, we will firstly show that $P^*$ is a cocycle, and admits a uniform attractor.
Then we prove that the uniform attractor of \eqref{eqG2.1}
converges to that of \eqref{eqG5_1} in probability measure space
when $\veps$ goes to zero. See Appendix \ref{shiftDS} for the definitions of
nonautonomous dynamical system, cocycle, uniform attractor  and
skew-product semiflow.

Define
$$B_{r}:=\left\{\mu\in Pr_{2}(L^{2}(\T^{d})):
\int_{L^{2}(\T^{d})}\|z\|^{2}\mu(\d z)\leq r^{2}\right\}
$$
for any $r>0$.
A subset $D\subset Pr_{2}(L^{2}(\T^{d}))$ is called {\em bounded}
if there exists a constant $r>0$ such that $D\subset B_{r}$.

By Lemma \ref{conpth}, the uniqueness in law of the solutions
for equation \eqref{maineq1} and the definition of $P^{*}$,
we have the following lemma, which states that $P^{*}$ is a cocycle.
\begin{lemma}\label{cocylem}
Consider equation \eqref{maineq1}. Let $\gamma(\cdot)\geq\frac{|\beta|}{\sqrt3}$.
Assume that {\rm(H$_f^1$)},  {\rm(H$_g^1$)} and  {\rm(H$_c$)} hold. Then
$P^{*}$ is a cocycle on $\left(\mathcal H(F),\R,\sigma\right)$ with
fiber $Pr_{2}(L^{2}(\T^{d}))$.
Moreover, the mapping given by
\[
\Pi :\R^{+}\times \mathcal H(F)\times Pr_{2}(L^{2}(\T^{d}))\rightarrow
\mathcal H(F)\times Pr_{2}(L^{2}(\T^{d})),
\]
\[
\Pi(t,(F,\mu)):=\left(\sigma_{t}F,P^{*}(t,F,\mu)\right)
\]
is a continuous skew-product semiflow.
\end{lemma}

\begin{remark}\label{uniestep}\rm
(i) It follows from Remark \ref{hulllem} that \eqref{ineq2}--\eqref{ineq4},
\eqref{BH01}, \eqref{BH02}, \eqref{B2p} and \eqref{gasmseq}
hold uniformly for all $\widetilde{F}\in \mathcal H(F)$.

(ii) Assume that $F:=(\gamma,f,g)$ satisfy (H$_f^1$), (H$_g^1$), (H$_g^2$),
(G$_\gamma$), (G$_f$), (G$_g^1$) and (G$_g^2$).
Then it is immediate to see that for any $\widetilde{F}\in\mathcal H(F)$, $\widetilde{\gamma}$,
$\widetilde{f}$ and $\widetilde{g}$ satisfy (G$_\gamma$),
(G$_f$), (G$_g^1$) and (G$_g^2$).
\end{remark}

For any given $\widetilde{F}\in \mathcal H(F)$,
under conditions of Proposition \ref{Boundedth}, equation
\eqref{maineq1} has a unique $\mcL^{2}(\Omega,\P;L^{2}(\T^{d}))$-bounded
solution $u_{\widetilde{F}}$ with the distribution
\[
\mu_{\widetilde{F}}(\cdot):=\mathscr L(u_{\widetilde{F}}(\cdot))
:\R\rightarrow Pr_2(L^2(\T^d)).
\]

\begin{prop}\label{invlemm}
Let $\gamma(\cdot)\geq\frac{|\beta|}{\sqrt3}$.
Consider equation \eqref{maineq1}. Suppose that
{\rm(H$_f^1$)}, {\rm(H$_f^3$)}, {\rm(H$_g^1$)},  {\rm(H$_g^2$)}
and {\rm(H$c$)} hold and $\lambda_{*}-\lambda_{f}-\frac{L_{g}^{2}}{2}>0$.
Then we have the following statements.
\begin{enumerate}
  \item Set
  $\mathfrak{A}_{\widetilde{F}}
  :=\overline{\left\{\mu_{\widetilde{F}}(t)
  \in Pr_{2}(L^{2}(\T^{d})):t\in \R\right\}}$. Then
  $$
  P^{*}(t,\widetilde{F},\mathfrak{A}_{\widetilde{F}})
  =\mathfrak{A}_{\sigma_{t}\widetilde{F}}
  $$
  for all $t\in \R^{+}$ and $\widetilde{F}\in \mathcal H(F)$.
  \item If $\mathcal H(F)$ is compact, then
  the skew product semiflow $\Pi$ admits a global attractor
  $\mathfrak{A}:=\omega\left(\mathcal H(F)\times\overline{
  \cup_{\widetilde{F}\in \mathcal H(F)}
  \mathfrak{A}_{\widetilde{F}}}\right)$.
  Moreover, $\Pi_{2}\mathfrak{A}$ is the uniform
  attractor of the cocycle $P^{*}$. Here
  $\Pi_{2}(\widetilde{F},\mu):=\mu$ for all
  $(\widetilde{F},\mu)\in \mathcal H(F)\times Pr_{2}(L^{2}(\T^{d}))$.
\end{enumerate}
\end{prop}

\begin{proof}
Similar to the proof of \cite[Proposition 5.12]{CLAM2021},
we have these conclusions.
\end{proof}

With the help of Theorem \ref{avethf}, Lemmas \ref{essol} and \ref{gasms}, and Proposition \ref{Boundedth},
we now establish the following global averaging principle. Since the argument is
similar to \cite[Theorem 5.14]{CLAM2021}, we omit the proof.
\begin{theorem}\label{gath}
Let $\gamma(\cdot)\geq\frac{|\beta|}{\sqrt3}$.
Consider equations \eqref{eqG2.1} and \eqref{eqG5_1}
under conditions {\rm(H$_f^1$)}--{\rm(H$_f^3$)}, {\rm(H$_g^1$)}--{\rm(H$_g^3$)},
{\rm(H$_c$)}, {\rm(G$_\gamma$)}, {\rm(G$_f$)} and {\rm(G$_g^1$)--(G$_g^2$)}. Assume that
$\lambda_{*}-\lambda_{f}-\frac{9}{2}L_{g}^{2}>0$.
If $\mathcal H(F)$ is compact, then
\begin{enumerate}
  \item the cocycle $P_{\veps}^{*}$ associated with
  stochastic CGL equation \eqref{eqG2.1} has a uniform attractor
  $\mathfrak{A}^{\veps}$ for any $0<\veps\leq1$;
  \item the cocycle $\bar{P}^{*}$ associated with averaged equation
        \eqref{eqG5_1} has a uniform attractor
        $\bar{\mathfrak{A}}$, which is a singleton set;
  \item
        \begin{equation*}%\label{gathlim}
        \lim_{\veps\rightarrow0}{\rm dist}_{Pr_{2}(L^{2}(\T^{d}))}
        \left(\mathfrak{A}^{\veps},\bar{\mathfrak{A}}\right)=0.
        \end{equation*}
\end{enumerate}
\end{theorem}

\begin{remark}\rm
Note that $\mathcal H(F)$ is compact
provided $F$ is Birkhoff recurrent, which includes periodic,
quasi-periodic, almost periodic, almost automorphic as special cases;
see e.g. \cite{Sel, sib}.
\end{remark}

%%%%%%%%%%%%%%%%%%%%%%%%%%%%%%%%%%%%%%%%%%%%%%
%% Single Appendix:                         %%
%%%%%%%%%%%%%%%%%%%%%%%%%%%%%%%%%%%%%%%%%%%%%%
\appendix

\section{}

\subsection{Dynamical system}\label{shiftDS}
Let $\left(\mathcal X,\rho\right)$ and
$\left(\mathcal Y,\rho_{1}\right)$ be two complete metric spaces.
Consider $C(\mathbb R,\mathcal X)$ equipped with the distance
\begin{equation*}\label{eqD1}
d(\varphi_{1},\varphi_{2}):=\sum_{k=1}^{\infty}\frac{1}{2^k}
\frac{d_{k}(\varphi_1,\varphi_{2})}{1+d_{k}(\varphi_1,\varphi_{2})},
\end{equation*}
where
$$
d_{k}(\varphi_1,\varphi_{2})
:=\sup\limits_{|t|\le k}\rho(\varphi_1(t),\varphi_{2}(t)),
$$
which generates the compact-open topology on $C(\mathbb R,\mathcal X)$.
Then the space $(C(\mathbb R,\mathcal X),d)$ is a complete metric space
(see e.g. \cite{Sel, Sch72, Sch85, sib}).

\begin{remark}[\cite{sib}]\label{remCh}\rm
Let $\{\varphi_{n}\}_{n=1}^{\infty}, \varphi\in C(\mathbb R,\mathcal X)$.
Then the following statements are equivalent.
\begin{enumerate}
  \item $\lim\limits_{n\to \infty}d(\varphi_{n},\varphi)=0$.
  \item $\lim\limits_{n\to\infty}\max\limits_{|t|\le l}
        \rho(\varphi_{n}(t),\varphi(t))=0$
         for any $l>0$.
  \item There exists a sequence $l_n\to +\infty$ such that
        $\lim\limits_{n\to \infty}\max\limits_{|t|\le l_n}
        \rho(\varphi_{n}(t),\varphi(t))=0$.
\end{enumerate}
\end{remark}

Now we recall some known definitions and results  in dynamical
systems (see e.g. \cite{CV2002, KR2011, Sel} for more details).
Let
$\left(\mathcal P, d_{\mathcal P}\right)$ be a metric space.
\begin{definition}\rm
A {\em nonautonomous dynamical system}
$\left(\sigma,\varphi\right)$ (in short, $\varphi$) consists
of two ingredients:
\begin{enumerate}
  \item A model of the nonautonomous driving system, namely a
  {\em dynamical system} $\sigma$ on $\mathcal P$ with time
      set $T=\mathbb Z$ or $\R$, i.e.
        \begin{enumerate}
          \item [(1)] $\sigma_{0}(\cdot)=Id_{\mathcal P}$,
          \item [(2)] $\sigma_{t+s}(p)=\sigma_{t}
          (\sigma_{s}(p))$ for all $t,s\in T$
          and $p\in \mathcal P$,
          \item [(3)] the mapping
          $(t,p)\mapsto\sigma_{t}(p)$ is continuous.
        \end{enumerate}
        If $T=\R$, $\sigma$ is called {\em flow} on $\mathcal P$;
        if $T=\R^{+}$, $\sigma$ is called {\em semiflow}
        on $\mathcal P$.
  \item A model of the nonautonomous perturbed system, namely
  a {\em cocycle} $\varphi:T^{+}\times\mathcal P\times \mathcal X
        \rightarrow \mathcal X$ satisfies
        \begin{enumerate}
          \item [(1)] $\varphi(0,p,x)=x$ for all
          $(p,x)\in \mathcal P\times\mathcal X$,
          \item [(2)] $\varphi(t+s,p,x)=\varphi(t,\sigma_{s}(p),
          \varphi(s,p,x))$ for all $s,t\in T^{+}$
          and $(p,x)\in\mathcal P\times\mathcal X$,
          \item [(3)] the mapping
          $(t,p,x)\mapsto\varphi(t,p,x)$ is continuous.
        \end{enumerate}
\end{enumerate}
Here $\mathcal P$ is called the {\em base} or {\em parameter space}
and $\mathcal X$ is the {\em fiber} or {\em state space}.
For convenience, we also write $\sigma_{t}(p)$ as $\sigma_{t}p$.
\end{definition}

\begin{definition}\rm
Let $(\sigma,\varphi)$ be a nonautonomous dynamical system with
base space $\mathcal P$ and fiber $\mathcal X$.
The semiflow $\Pi: T^{+}\times\mathcal P\times\mathcal X\rightarrow
\mathcal P\times \mathcal X$ defined by
\[
\Pi(t,(p,x)):=\left(\sigma_{t}p,\varphi(t,p,x)\right)
\]
is called {\em skew product semiflow}.
\end{definition}

\begin{definition}\rm
Define $\mathfrak{X}:=\mathcal P\times\mathcal X$.
We say that a nonempty compact
subset $\mathfrak{A}\subset\mathfrak{X}$ is
{\em global attractor} for skew product semiflow $\Pi$, if
\begin{enumerate}
  \item $\Pi(t,\mathfrak{A})=\mathfrak{A}$ for all
  $t\in\mathbb R^{+}$,
  \item $\lim\limits_{t\rightarrow+\infty}{\rm dist}_{\mathfrak{X}}
  \left(\Pi(t,D),\mathfrak{A}\right)=0$ for every nonempty bounded
  subset $D\subset\mathfrak{X}$,
\end{enumerate}
where ${\rm dist}_{\mathfrak{X}}(A,B)$ is the Hausdorff semi-metric
between sets $A$ and $B$, i.e.
${\rm dist}_{\mathfrak{X}}(A,B):=\sup\limits_{x\in A}d(x,B)$  with
$d(x,B):=\inf\limits_{y\in B}d_{\mathfrak{X}}(x,y)$. Here
$d_{\mathfrak{X}}(x,y)=d_{\mathcal P}(p_{1},p_{2})+\rho(x_{1},x_{2})$
for all $x:=(p_{1},x_{1}),y:=(p_{2},x_{2})\in \mathcal P\times\mathcal X$.
\end{definition}

\iffalse
\begin{lemma}[see e.g. \cite{CV2002}]\label{conta}
Let $\{S(t)\}_{t\geq0}$ be a  semiflow in a complete metric space
$\mathcal X$. If $\{S(t)\}_{t\geq0}$ has a compact attracting set $K\subset\mathcal X$,
i.e.
\[
\lim_{t\rightarrow+\infty}{\rm{dist_{\mathcal X}}}(S(t)B,K)=0
\]
for all bounded subset $B\subset\mathcal X$, then $\{S(t)\}_{t\geq0}$
admits a global attractor $\mathcal A:=\omega(K)$, where $\omega(K)$
is the $\omega$-limit set of $K$, i.e.
$\omega(K):=\cap_{t\geq0}\overline{\cup_{s\geq t}S(s)K}$.
\end{lemma}
\fi

\begin{definition}\label{gadef}\rm
We say that a compact set $\mathcal A\subset\mathcal X$ is the
{\em uniform attractor (with respect to $p\in\mathcal P$)} of cocycle
$\varphi$ if the following conditions are fulfilled:
\begin{enumerate}
  \item The set $\mathcal A$ is uniformly attracting, i.e.
  \[
  \lim\limits_{t\rightarrow+\infty}\sup\limits_{p\in\mathcal P}
  {\rm dist}_{\mathcal X}\left(\varphi(t,p,B),\mathcal A\right)=0
  \]
  for all bounded subset $B\subset\mathcal X$.
  \item If $\mathcal A_{1}$ is another closed uniformly attracting
  set, then $\mathcal A\subset\mathcal A_{1}$.
\end{enumerate}
\end{definition}

Let us now introduce a shift dynamical system.
For any
$(\tau,\varphi)\in\R\times C(\R,\mathcal X)$, define the mapping
$\sigma:\R\times C(\R,\mathcal X)\rightarrow C(\R,\mathcal X)$
by $\sigma(\tau,\varphi):=\varphi^{\tau}$. Then the triplet
$\left(C(\R,\mathcal X),\R,\sigma\right)$ is a dynamical system which is
called {\em shift dynamical system} or {\em the Bebutov dynamical system}.
Indeed, it is easy to check that $\sigma(0,\varphi)=\varphi$ and
$\sigma(\tau_{1}+\tau_{2},\varphi)
=\sigma(\tau_{2},\sigma(\tau_{1},\varphi))$ for any
$\varphi\in C(\R,\mathcal X)$ and $\tau_{1},\tau_{2}\in\R$. And it can be
proved that the mapping
$\sigma:\R\times C(\R,\mathcal X)\rightarrow C(\R,\mathcal X)$
is continuous; see e.g. \cite{Sel, Sch72, sib}.
Note that $\mathcal H(\varphi)\subset C(\R,\mathcal X)$ is closed and
translation invariant. Then it naturally defines on $\mathcal H(\varphi)$
a shift dynamical system $\left(\mathcal H(\varphi),\R,\sigma\right)$.

Denote by $BUC(\R\times\mathcal X,\mathcal Y)$ (see \cite{CL_2017}) the space of
all continuous functions $f:\R\times\mathcal X\rightarrow\mathcal Y$
satisfying the following conditions:
\begin{enumerate}
  \item $f$ is bounded on every bounded subset from $\R\times\mathcal X$;
  \item $f$ is continuous in $t\in\R$ uniformly with respect to $x$
  on each bounded subset $Q\subset\mathcal X$.
\end{enumerate}
We endow $BUC(\R\times\mathcal X,\mathcal Y)$ with the following distance
\begin{equation}\label{dBUC}
d(f,g):=\sum_{k=1}^{\infty}\frac{1}{2^{k}}\frac{d_{k}(f,g)}{1+d_{k}(f,g)},
\end{equation}
where $d_{k}(f,g):=\sup\limits_{|t|\leq k,x\in Q_{k}}\rho_{1}(f(t,x),g(t,x))$.
Here $Q_{k}\subset\mathcal X$ is bounded,
$Q_{k}\subset Q_{k+1}$ and $\cup_{k\in\N}Q_{k}=\mathcal X$.
We notice that $d$ generates the topology of uniform convergence
on  bounded subsets of $\R\times\mathcal X$
and $\left(BUC(\R\times\mathcal X,\mathcal Y),d\right)$
is a complete metric space.

Let $f\in BUC(\R\times\mathcal X,\mathcal Y)$ and $\tau\in\R$.
Note that $BUC(\R\times\mathcal X,\mathcal Y)$ is invariant with respect to
translations. Define a mapping
$\sigma:\R\times BUC(\R\times\mathcal X,\mathcal Y)
\rightarrow BUC(\R\times\mathcal X,\mathcal Y)$, $(\tau,f)\mapsto f^{\tau}$.
Then it can be proved that the triplet
$\left(BUC(\R\times\mathcal X,\mathcal Y),\R,\sigma\right)$
is a dynamical system.
%See Chapter I in \cite{Ch2015} for details.
Given $f\in BUC(\R\times\mathcal X,\mathcal Y)$,
$\mathcal H(f)\subset BUC(\R\times\mathcal X,\mathcal Y)$ is closed
and translation invariant. Consequently, it naturally defines
on $\mathcal H(f)$ a shift dynamical system
$\left(\mathcal H(f),\R,\sigma\right)$.

We write $BC(\mathcal X,\mathcal Y)$ to mean the space of
all continuous functions $f:\mathcal X\rightarrow\mathcal Y$
which are bounded on every bounded subset of  $\mathcal X$
and equipped with the following metric
\[
d(f,g):=\sum_{k=1}^{\infty}\frac{1}{2^{k}}
\frac{d_{k}(f,g)}{1+d_{k}(f,g)},
\]
where $d_{k}(f,g):=\sup\limits_{x\in Q_{k}}\rho_{1}(f(x),g(x))$.
Note that $\left(BC(\mathcal X,\mathcal Y),d\right)$
is a complete metric space. For any
$f\in BUC(\R\times\mathcal X,\mathcal Y)$, define the mapping
$\mathcal F:\R\rightarrow BC(\mathcal X,\mathcal Y)$ by
$\mathcal F(t):=f(t,\cdot):\mathcal X\rightarrow\mathcal Y$.
Clearly, $\mathcal F\in C(\R,BC(\mathcal X,\mathcal Y))$.

\begin{definition}\label{defF1} \rm
\begin{enumerate}
\item We say that a function $\varphi\in C(\R,\mathcal X)$
  {\em possesses the property A} if the motion $\sigma(\cdot,\varphi)$
  through $\varphi$ with respect to the Bebutov dynamical system
  $(C(\R\times\mathcal X),\R,\sigma)$ possesses the property A.
\item Similarly, we say that $f\in BUC(\R\times\mathcal X,\mathcal Y)$
  {\em possesses the property A in $t\in\R$ uniformly with respect to
  $x$ on each bounded subset $Q\subset\mathcal X$}, if the motion
  $\sigma(\cdot,f):\R\rightarrow BUC(\R\times\mathcal X,\mathcal Y)$
  through $f$ with respect to the Bebutov dynamical system
  $\left(BUC(\R\times\mathcal X,\mathcal Y),\R,\sigma\right)$
  possesses the property A.
\end{enumerate}
Here the property A may be stationary, periodic, Bohr/Levitan almost
periodic, etc.
\end{definition}

\subsection{Proof of Theorem \ref{wellD}}\label{eupf}
\begin{proof}
We set $f^{n}(t,u):=P_{n}f(t,u)$ and $g^{n}(t,u):=P_{n}g(t,u)$.
Let $\{w_{k},k\in\mathbb N\}$ be an orthonormal basis of $U$ and set
$$
W^{n}(\cdot):=\sum\limits_{k=1}^{n}\langle W(\cdot),w_{i}\rangle_{U}w_{i}.
$$
Let $u^n(t),t\geq s$ be the solutions to
the following  finite dimensional equations
\begin{equation}\label{GFeq}
  \left\{
   \begin{aligned}
   \ \d u^{n}(t)=&
   \left[(1+i\alpha )\Delta u^{n}(t)
   -(\gamma(t)+i\beta)P_{n}|u^{n}(t)|^{2}u^{n}(t)
    +P_nf(t,u^{n}(t))\right]\d t\\
   &
    +P_ng(t,u^{n}(t))\d W^{n}(t)\\
  \ u^{n}(s)=&P_{n}\zeta_{s}.
   \end{aligned}
   \right.
\end{equation}
It follows from \cite[Theorem 3.1.1]{LR2015} that there exists a
unique solution $u^{n}(t),t\geq s$ to \eqref{GFeq}
for any $n\in\mathbb N$.
Similar to the proof of \eqref{uniest1}, we have
\begin{align}\label{uniestGF}
&
\mathbb E\left(\sup_{s\leq t\leq s+T}\|u^{n}(t)\|^{2p}\right)
+\mathbb E\int_{s}^{s+T}\|u^n(t)\|^{2p-2}\|u^{n}(t)\|_{1}^{2}\d t\\\nonumber
&
+\mathbb E\int_{s}^{s+T}\|u^n(t)\|^{2p-2}\|u^{n}(t)\|_{L^{4}(\T^{d})}^{4}\d t
\leq C_{T}(1+\mathbb E\|\zeta_{s}\|^{2p})
\end{align}
for all $T>0$, where $p\geq1$, and $C_{T}$ is independent of $n$.

Then there exists a subsequence of $\{u^{n}\}$,
which we still denote by $\{u^{n}\}$, such that
\begin{enumerate}
  \item[(1)]
  $u^{n}\rightarrow u$ weakly in
  $L^{2}([s,s+T]\times\Omega,\d t\otimes\mathbb P;
  L^{2}(\T^{d}))$, $L^{2}([s,s+T]\times\Omega,\d t\otimes
  \mathbb P; H_0^{1})$ and $L^{4}([s,s+T]\times\Omega,\d t\otimes
  \mathbb P;L^{4}(\T^{d}))$.
  \item[(2)]
  $|u^{n}(\cdot)|^{2}u^{n}(\cdot)\rightarrow Y^{1}(\cdot)$
  weakly in $L^{\frac{4}{3}}\left([s,s+T]\times\Omega,\d t
  \otimes\mathbb P;L^{\frac{4}{3}}(\T^{d})\right)$.
  \item[(3)]
  $f^n(\cdot,u^{n}(\cdot))\rightarrow Y^{2}(\cdot)$ weakly in
  $L^{2}\left([s,s+T]\times\Omega,\d t \otimes\mathbb P;
  L^{2}(\T^{d})\right)$.
  \item[(4)]
  $g^{n}(\cdot,u^{n}(\cdot))\rightarrow Z(\cdot)$ weakly in
  $L^{2}\left([s,s+T]\times\Omega,\d t\otimes\mathbb P;L_{2}(U,
  L^{2}(\T^{d}))\right)$ and hence
  \[
  \int_{s}^{t}g^{n}(\sigma,u^{n}(\sigma))\d W^{n}(\sigma)
  \rightarrow\int_{s}^{t}Z(\sigma)\d W(\sigma)
  \]
  weakly* in $L^{\infty}\left([s,s+T],\d t;L^{2}(\Omega,\mathbb P;
  L^{2}(\T^{d}))\right)$.
\end{enumerate}
For all $v\in \cup_{n\geq1}H_{n}$,
$\phi\in L^{\infty}([s,s+T]\times\Omega,\d t\otimes\mathbb P;\R)$,
it follows from Fubini's theorem that
\begin{align*}
&
\mathbb E\left(\int_{s}^{s+T}\langle u(t),\phi(t)v\rangle
\d t\right)\\
&
=\lim\limits_{n\rightarrow\infty}\mathbb E\left(\int_{s}^{s+T}
\langle u^{n}(t),\phi(t)v\rangle\d t\right)\\
&
=\lim\limits_{n\rightarrow\infty}\mathbb E\Bigg(\int_{s}^{s+T}
\left\langle P_{n}\zeta_{s}+\int_{s}^{t}\left[(1+i\alpha)\Delta
u^{n}(\sigma)-(\gamma(\sigma)+i\beta)P_{n}|u^{n}(\sigma)|^{2}u^{n}(\sigma)
\right]\d\sigma,\phi(t)v\right\rangle\d t\\
&\qquad
+\left\langle\int_{s}^{t}f^{n}(\sigma,u^{n}(\sigma))\d\sigma
+\int_{s}^{t}g^{n}(\sigma,u^{n}(\sigma))\d W^{n}(\sigma),
\phi(t)v\right\rangle\d t\Bigg)\\
&
=\mathbb E\left\langle \zeta_{s},\int_{s}^{s+T}\phi(t)v
\d t\right\rangle
+\mathbb E\left(\int_{s}^{s+T}\left\langle\int_{s}^{t}
(1+i\alpha)u(\sigma)\d\sigma,\phi(t)\Delta v\right\rangle\d t\right)
\\
&\quad
+\mathbb E\left(\int_{s}^{s+T}\left\langle\int_{s}^{t}
\left(-(\gamma(\sigma)+i\beta)Y^{1}(\sigma)+Y^{2}(\sigma)\right)\d\sigma,
\phi(t)v\right\rangle\d t\right)\\
&\quad
+\mathbb E\left(\int_{s}^{s+T}\left\langle\int_{s}^{t}Z(\sigma)
\d W(\sigma),\phi(t)v\right\rangle\d t\right).
\end{align*}
Therefore, we have
\[
u(t)=\zeta_{s}+\int_{s}^{t}\left((1+i\alpha)\Delta u(\sigma)
-(\gamma(\sigma)+i\beta)Y^{1}(\sigma)+Y^{2}(\sigma)\right)\d t
+\int_{s}^{t}Z(\sigma)\d W(\sigma) \quad
\d t\otimes\mathbb P-{\rm a.s.}
\]
Taking $V=L^4(\T^d)\cap H_0^1$ and $H:=L^2(\T^d)$ in \cite[Theorem 4.5]{LR2015},
it follows that $u$ is a continuous
$L^2(\T^d)$-valued $\mathcal F_{t}$-adapted process.

Now we prove that $-(\gamma(\cdot)+i\beta)Y^{1}+Y^{2}=-(\gamma(\cdot)+i\beta)|u|^{2}u+f(\cdot,u)$ and
$Z=g(\cdot,u)$  $\d t\otimes\mathbb P$-a.s.
Note that for any nonnegative $\psi\in L^{\infty}([s,s+T],\d t;\R)$ we have
\begin{align*}
&
\mathbb E\left(\int_{s}^{s+T}\psi(t)\|u(t)\|^{2}\d t\right)\\
&
=\lim\limits_{n\rightarrow\infty}\mathbb E\left(\int_{s}^{s+T}\left
\langle\psi(t)u(t),u^{n}(t)\right\rangle\d t\right)\\
&
\leq\left(\mathbb E\int_{s}^{s+T}\psi(t)\|u(t)\|^{2}\d t
\right)^{\frac{1}{2}}\liminf\limits_{n\rightarrow\infty}\left(
\mathbb E\int_{s}^{s+T}\psi(t)\|u^{n}(t)\|^{2}\d t
\right)^{\frac{1}{2}}<\infty.
\end{align*}
Then
\begin{equation}\label{exeq1}
\mathbb E\left(\int_{s}^{s+T}\psi(t)\|u(t)\|^{2}\d t\right)
\leq\liminf\limits_{n\rightarrow\infty}
\mathbb E\left(\int_{s}^{s+T}\psi(t)\|u^{n}(t)\|^{2}\d t\right).
\end{equation}
According to the product rule and It\^o's formula we get
\begin{align}\label{exeq2}
&
\mathbb E\left(e^{-c(t-s)}\|u(t)\|^{2}\right)
-\mathbb E\|u(s)\|^{2}\\\nonumber
&
=\mathbb E\Bigg(\int_{s}^{t}e^{-c(\sigma-s)}\bigg(2\left\langle
(1+i\alpha)\Delta u(\sigma)-(\gamma(\sigma)+i\beta)Y^{1}(\sigma)+Y^{2}(\sigma),
u(\sigma)\right\rangle\\\nonumber
&\qquad
+\|Z(\sigma)\|_{L_{2}(U,L^{2}(\T^{d}))}^{2}
-c\|u(\sigma)\|^{2}\bigg)\d\sigma\Bigg)
\end{align}
for any constant $c$. Let $c=2L_{f}+L_{g}^{2}$ and
\[
K_{i}:=L^{2}\left([s,s+T]\times\Omega,\d t\otimes\mathbb P;
W_{i}\right),\quad
K_{3}:=L^{4}\left([s,s+T]\times\Omega,\d t\otimes\mathbb P;
L^{4}(\T^{d})\right),
\]
where $i=1,2$, $W_{1}:=L^{2}(\T^{d})$ and $W_{2}:=H_0^{1}$.
Note that $\la P_nv,w\ra=\la v,w\ra$ for all $v\in L^2(\T^d)$ and $w\in H_n$.
Hence, by (H$_f^1$) and (H$_g^1$), for any $\phi\in K_{1}\cap K_{2}\cap K_{3}$ we obtain
\begin{align*}
&
\mathbb E\left(e^{-c(t-s)}\|u^{n}(t)\|^{2}\right)
-\mathbb E\|u^{n}(s)\|^{2}\\\nonumber
&
\leq\mathbb E\Bigg(\int_{s}^{t}e^{-c(\sigma-s)}\bigg(
2\left\langle(1+i\alpha)\Delta (u^{n}(\sigma)-\phi(\sigma)),
u^{n}(\sigma)-\phi(\sigma)\right\rangle\\\nonumber
&\qquad
-2\left\langle(\gamma(\sigma)+i\beta)\left(|u^{n}(\sigma)|^{2}u^{n}(\sigma)
-|\phi(\sigma)|^{2}\phi(\sigma)\right),
u^{n}(\sigma)-\phi(\sigma)\right\rangle\\\nonumber
&\qquad
+2\left\langle f(\sigma,u^{n}(\sigma))
-f(\sigma,\phi(\sigma)),u^{n}(\sigma)-\phi(\sigma)
\right\rangle\\\nonumber
&\qquad
+\|g(\sigma,u^{n}(\sigma))-g(\sigma,\phi(\sigma))
\|_{L_{2}(U,L^{2}(\T^{d}))}^{2}
-c\|u^{n}(\sigma)-\phi(\sigma)\|^{2}\bigg)\d\sigma\Bigg)\\\nonumber
&\quad
+\mathbb E\Bigg(\int_{s}^{t}e^{-c(\sigma-s)}\bigg(
2\left\langle(1+i\alpha)\Delta \phi(\sigma),
u^{n}(\sigma)\right\rangle
+2\left\langle(1+i\alpha)\Delta (u^{n}(\sigma)-\phi(\sigma)),
\phi(\sigma)\right\rangle\\\nonumber
&\qquad
-2\left\langle(\gamma(\sigma)+i\beta)|\phi(\sigma)|^{2}\phi(\sigma),
u^{n}(\sigma)\right\rangle
-2\left\langle(\gamma(\sigma)+i\beta)\left(|u^{n}(\sigma)|^{2}u^{n}(\sigma)
-|\phi(\sigma)|^{2}\phi(\sigma)\right),\phi(\sigma)\right\rangle
\\\nonumber
&\qquad
+2\left\langle f(\sigma,\phi(\sigma)),u^{n}(\sigma)\right\rangle
+2\left\langle f(\sigma,u^{n}(\sigma))
-f(\sigma,\phi(\sigma)),\phi(\sigma)\right\rangle
-\|g(\sigma,\phi(\sigma))\|_{L_{2}(U,L^{2}(\T^{d}))}^{2}
\\\nonumber
&\qquad
+2\left\langle g(\sigma,u^{n}(\sigma)),g(\sigma,\phi(\sigma))
\right\rangle_{L_{2}(U,L^{2}(\T^{d}))}
-2c\langle u^{n}(\sigma),\phi(\sigma)\rangle
+c\|\phi(\sigma)\|^{2}
\bigg)\d\sigma\Bigg)\\\nonumber
&
\leq\mathbb E\Bigg(\int_{s}^{t}e^{-c(\sigma-s)}\bigg(
2\left\langle(1+i\alpha)\Delta \phi(\sigma),
u^{n}(\sigma)\right\rangle
+2\left\langle(1+i\alpha)\Delta (u^{n}(\sigma)-\phi(\sigma)),
\phi(\sigma)\right\rangle\\\nonumber
&\qquad
-2\left\langle(\gamma(\sigma)+i\beta)|\phi(\sigma)|^{2}\phi(\sigma),
u^{n}(\sigma)\right\rangle
-2\left\langle(\gamma(\sigma)+i\beta)\left(|u^{n}(\sigma)|^{2}u^{n}(\sigma)
-|\phi(\sigma)|^{2}\phi(\sigma)\right),\phi(\sigma)\right\rangle
\\\nonumber
&\qquad
+2\left\langle f(\sigma,\phi(\sigma)),u^{n}(\sigma)\right\rangle
+2\left\langle f(\sigma,u^{n}(\sigma))
-f(\sigma,\phi(\sigma)),\phi(\sigma)\right\rangle
-\|g(\sigma,\phi(\sigma))\|_{L_{2}(U,L^{2}(\T^{d}))}^{2}
\\\nonumber
&\qquad
+2\left\langle g(\sigma,u^{n}(\sigma)),g(\sigma,\phi(\sigma))
\right\rangle_{L_{2}(U,L^{2}(\T^{d}))}
-2c\langle u^{n}(\sigma),\phi(\sigma)\rangle
+c\|\phi(\sigma)\|^{2}
\bigg)\d\sigma\Bigg).\\\nonumber
\end{align*}
Letting $n\rightarrow\infty$, in view of \eqref{exeq1} we have
\begin{align*}
&
\mathbb E\left(\int_{s}^{s+T}\psi(t)\left(e^{-c(t-s)}\|u(t)\|^{2}
-\|u(s)\|^{2}\right)\d t\right)\\\nonumber
&
\leq\mathbb E\Bigg(\int_{s}^{s+T}\psi(t)\bigg(\int_{s}^{t} e^{-c(\sigma-s)}\Big[
2\left\langle(1+i\alpha)\Delta \phi(\sigma),u(\sigma)\right\rangle
+2\left\langle(1+i\alpha)\Delta (u(\sigma)-\phi(\sigma)),
\phi(\sigma)\right\rangle\\\nonumber
&\qquad
-2\left\langle(\gamma(\sigma)+i\beta)|\phi(\sigma)|^{2}\phi(\sigma),
u(\sigma)\right\rangle
-2\left\langle(\gamma(\sigma)+i\beta)\left(Y^{1}(\sigma)
-|\phi(\sigma)|^{2}\phi(\sigma)\right),\phi(\sigma)\right\rangle
\\\nonumber
&\qquad
+2\left\langle f(\sigma,\phi(\sigma)),u(\sigma)\right\rangle
+2\left\langle Y^{2}(\sigma)-f(\sigma,\phi(\sigma)),
\phi(\sigma)\right\rangle
-\|g(\sigma,\phi(\sigma))\|_{L_{2}(U,L^{2}(\T^{d}))}^{2}\\\nonumber
&\qquad
+2\left\langle Z(\sigma),g(\sigma,\phi(\sigma))
\right\rangle_{L_{2}(U,L^{2}(\T^{d}))}
-2c\langle u(\sigma),\phi(\sigma)\rangle
+c\|\phi(\sigma)\|^{2}
\Big]\d\sigma\bigg)\d t\Bigg).
\end{align*}
Combining this with \eqref{exeq2}, we get
\begin{align}\label{exeq10}
&
\mathbb E\Bigg(\int_{s}^{s+T}\psi(t)\bigg(\int_{s}^{t} e^{-c(\sigma-s)}\Big[
-2\left\langle (\gamma(\sigma)+i\beta)\left(Y^{1}(\sigma)-|\phi(\sigma)|^{2}
\phi(\sigma)\right),u(\sigma)-\phi(\sigma)\right\rangle\\\nonumber
&\qquad
+2\left\langle Y^{2}(\sigma)-f(\sigma,\phi(\sigma)),
u(\sigma)-\phi(\sigma)\right\rangle\\\nonumber
&\qquad
+\|g(\sigma,\phi(\sigma))-Z(\sigma)
\|_{L_{2}(U,L^{2}(\T^{d}))}^{2}
-c\|u(\sigma)-\phi(\sigma)\|^{2}\Big]\d\sigma\bigg)\d t\Bigg)
\leq0.
\end{align}
Letting $\phi=u$ in \eqref{exeq10}, we have $Z=g(\cdot,u)$
$\d t\otimes\mathbb P$-a.s.
Then letting $\phi=u-\veps \widetilde{\phi}v$ for $\veps>0$
, $v\in H_0^1$ and $\widetilde{\phi}\in L^{\infty}([s,s+T]\times\Omega,\d t\times
\mathbb P;\R)$ we have
\begin{align*}
&
\mathbb E\Bigg(\int_{s}^{s+T}\psi(t)\bigg(\int_{s}^{t} e^{-c(\sigma-s)}\Big[
2\left\langle Y^{2}(\sigma)
-f(\sigma,u(\sigma)-\veps\widetilde{\phi}(\sigma)v),
\veps\widetilde{\phi}(\sigma)v\right\rangle-c\veps^{2}
\|\widetilde{\phi}v\|^{2}\\
&\quad
-2\left\langle (\gamma(\sigma)+i\beta)\left(Y^{1}(\sigma)
-|u(\sigma)-\veps\widetilde{\phi}(\sigma)v|^{2}
(u(\sigma)-\veps\widetilde{\phi}(\sigma)v)\right),
\veps\widetilde{\phi}(\sigma)v\right\rangle\Big]\d\sigma\bigg)\d t\Bigg)\leq0.
\end{align*}
Dividing both sides by $\veps$ and letting $\veps\rightarrow0$,
in view of Lebesgue's dominated convergence theorem, we obtain
\begin{align*}
&
\mathbb E\Bigg(\int_{s}^{s+T}\psi(t)\bigg(\int_{s}^{t} e^{-c(\sigma-s)}\Big[
-\left\langle (\gamma(\sigma)+i\beta)\left(Y^{1}(\sigma)
-|u(\sigma)|^{2}u(\sigma)\right),\widetilde{\phi}(\sigma)v
\right\rangle\\\nonumber
&\quad
+\left\langle Y^{2}(\sigma)
-f(\sigma,u(\sigma)),\widetilde{\phi}(\sigma)v\right\rangle
\Big]\d\sigma\bigg)\d t\Bigg)\leq0.
\end{align*}
Therefore, we have $-(\gamma(\cdot)+i\beta)Y^{1}+Y^{2}=-(1+i\beta)|u|^{2}u+f(\cdot,u)$
$\d t\otimes\mathbb P$-a.s.

Now we show the uniqueness of solutions. Suppose that there exist two solutions $u_{1}(t):=u_1(t,s,\zeta_s)$ and $u_{2}(t):=u_2(t,s,\zeta_s)$ for all $t\geq s$, then
\begin{align*}
&
\mathbb E\left({\rm{e}}^{(2\lambda_{*}-2L_{f}-L_{g}^{2})(t-s)}
\|u_{1}(t)-u_{2}(t)\|^{2}\right)\\
&
=\int_{s}^{t}(2\lambda_{*}-2L_{f}-L_{g}^{2})
{\rm{e}}^{(2\lambda_{*}-2L_{f}-L_{g}^{2})(\sigma-s)}
\mathbb E\|u_{1}(\sigma)-u_{2}(\sigma)\|^{2}\d\sigma \\
& \quad
+\mathbb E\int^{t}_{s}{\rm{e}}^{(2\lambda_{*}-2L_{f}
-L_{g}^{2})(\sigma-s)}\Big(2\langle (1+i\alpha)
\Delta\left(u_{1}(\sigma)-u_{2}(\sigma)\right),
u_{1}(\sigma)-u_{2}(\sigma)\rangle\\
&\qquad
-2\langle(\gamma(\sigma)+i\beta)\left(|u_{1}(\sigma)|^{2}
u_{1}(\sigma)
-|u_{2}(\sigma)|^{2}u_{2}(\sigma)\right),
u_{1}(\sigma)-u_{2}(\sigma)\rangle\\
&\qquad
+2\langle f(\sigma,u_{1}(\sigma))
-f(\sigma,u_{2}(\sigma)),
u_{1}(\sigma)-u_{2}(\sigma)\rangle\\
&\qquad
+\|g(\sigma,u_{1}(\sigma))
-g(\sigma,u_{2}(\sigma))\|^{2}_{L_{2}(U,L^{2}(\T^{d}))}\Big)\d\sigma\leq0.
\end{align*}
\end{proof}

%%%%%%%%%%%%%%%%%%%%%%%%%%%%%%%%%%%%%%%%%%%%%%
%% Multiple Appendixes:                     %%
%%%%%%%%%%%%%%%%%%%%%%%%%%%%%%%%%%%%%%%%%%%%%%
%\begin{appendix}
%\section{???}
%
%\section{???}
%
%\end{appendix}

%%%%%%%%%%%%%%%%%%%%%%%%%%%%%%%%%%%%%%%%%%%%%%
%% Support information, if any,             %%
%% should be provided in the                %%
%% Acknowledgements section.                %%
%%%%%%%%%%%%%%%%%%%%%%%%%%%%%%%%%%%%%%%%%%%%%%
\section*{Acknowledgements}
The first author would like to acknowledge the
support from China Scholarship Council and
warm hospitality of Bielefeld University.
The second author was supported by NSFC Grants
11871132, 11925102,
Dalian High-level Talent Innovation Project (Grant 2020RD09),
and Xinghai Jieqing fund from Dalian University of
Technology.
The third author was supported by the Deutsche Forschungsgemeinschaft
(DFG, German Research Foundation) - SFB 1283/2 2021 - 317210226.

\end{document}